\documentclass[11pt,reqno,a4paper]{article}
 \usepackage[english]{babel}
\usepackage[left=2cm,right=2cm,top=2cm,bottom=2cm]{geometry}
\usepackage{amsmath,mathrsfs,stmaryrd}
\usepackage{amsthm}
\usepackage{mathtools}
\usepackage{empheq}
\usepackage{tabularx}
\usepackage{enumerate,enumitem}
\usepackage{dsfont}
\usepackage[hidelinks]{hyperref}
\usepackage{csquotes}
\usepackage{xcolor}
\usepackage{times}
\usepackage{cancel}
\usepackage{cleveref}
\usepackage{verbatim}

\makeatletter
\@addtoreset{equation}{section}
\makeatother

\theoremstyle{plain}
\newtheorem{theorem}{Theorem}[section] 
\newtheorem*{theorem*}{Theorem}
\newtheorem{proposition}[theorem]{Proposition}
\newtheorem*{proposition*}{Proposition}
\newtheorem{corollary}[theorem]{Corollary}
\newtheorem*{corollary*}{Corollary}
\newtheorem{assumption}[theorem]{Assumption}
\newtheorem{lemma}[theorem]{Lemma}
\newtheorem*{lemma*}{Lemma}

\theoremstyle{definition}
\newtheorem{remark}[theorem]{Remark}

\newtheorem{definition}[theorem]{Definition}

\newcommand{\R}{\ensuremath{\mathbb R}} 

\newcommand{\dd}{\ensuremath{{\hspace{0.01em}\mathrm d}}}

\newcommand{\V}{\ensuremath{\mathcal V}}

\usepackage{amssymb,graphicx}
\newcommand{\bbGamma}{{\mathpalette\makebbGamma\relax}}
\newcommand{\makebbGamma}[2]{%
	\raisebox{\depth}{\scalebox{1}[-1]{$\mathsurround=0pt#1\mathbb{L}$}}%
}
\definecolor{cadmiumgreen}{rgb}{0.0, 0.42, 0.24}
\definecolor{tangelo}{rgb}{0.98, 0.3, 0.0}
\definecolor{shamrockgreen}{rgb}{0.0, 0.62, 0.38}
\definecolor{darkolivegreen}{rgb}{0.33, 0.42, 0.18}
\definecolor{deepmagenta}{rgb}{0.8, 0.0, 0.8}
\definecolor{burgundy}{rgb}{0.5, 0.0, 0.13}
\definecolor{darkbyzantium}{rgb}{0.36, 0.22, 0.33}

\usepackage{authblk}

\usepackage[
hidelinks 
]{hyperref}
\newcommand\blfootnote[1]{%
	\begingroup
	\renewcommand\thefootnote{}\footnote{#1}%
	\addtocounter{footnote}{-1}%
	\endgroup
}

\begin{document}

		\title{On ergodic invariant measures for the stochastic Landau-Lifschitz-Gilbert equation in 1D} 
		 \author{Emanuela Gussetti}
		 \affil{\small Bielefeld University, Germany \vspace{-1em} }
	\maketitle
	
	\begin{abstract}
		We establish existence of an ergodic invariant measure on $H^1(D,\mathbb{R}^3)\cap L^2(D,\mathbb{S}^2)$ for the stochastic Landau-Lifschitz-Gilbert equation on a bounded one dimensional interval $D$. The conclusion is achieved by employing the classical Krylov-Bogoliubov theorem. In contrast to other equations, verifying the hypothesis of the Krylov-Bogoliubov theorem is not a standard procedure. We employ rough paths theory to show that the semigroup associated to the equation has the Feller property in $H^1(D,\mathbb{R}^3)\cap L^2(D,\mathbb{S}^2)$. It does not seem possible to achieve the same conclusion by the classical Stratonovich calculus. On the other hand, we employ the classical Stratonovich calculus to prove the tightness hypothesis. The Krein-Milman theorem implies existence of an ergodic invariant measure.
		In case of spatially constant noise, we show that there exists a unique Gibbs invariant measure and we establish the qualitative behaviour of the unique stationary solution. In absence of the anisotropic energy and for a spatially constant noise, we are able to provide a pathwise long time behaviour result: in particular, every solution synchronises with a spherical Brownian motion and it is recurrent for large times.
	\end{abstract}

    \section{Introduction }
    The Landau-Lifschitz-Gilbert (LLG) equation is a phenomenological model that describes the time evolution of the magnetization of a ferromagnetic material for small temperatures below the Curie temperature, namely the temperature above which the material loses its permanent magnetic properties. The LLG equation is given by
    \begin{equation}\label{LLG_intro}
    \dd u =
    \lambda_1  u\times \mathcal{H} (u)
    -\lambda_2 u\times (u\times \mathcal{H} (u))\, ,
    \end{equation}
    where $\lambda_1\in\mathbb{R}$ and $\lambda_2>0$. Before introducing $\mathcal{H}(\phi)$, we need to define the energy functional. Given a bounded interval $D\subset \mathbb{R}$, we define the energy functional for all $\phi\in L^2(D;\mathbb{S}^2)\cap H^1(D;\mathbb{R}^3)$  and by
    \begin{align}\label{eq:energy_intro}
    	\mathcal{E}(\phi):=\frac{1}{2}\int_{D}\left[ |\partial_x \phi(x)|^2+g(\phi(x))\cdot \phi(x)\right] \dd x\, ,
    \end{align}
    where $g:\mathbb{R}^3\rightarrow \mathbb{R}^3$ is a linear function. The gradient part in \eqref{eq:energy_intro} represents the exchange energy, which is a quantum mechanical effect responsible of the ferromagnetic phenomenon. The bilinear term in \eqref{eq:energy_intro} takes into account the anisotropic energy, which describes the fact that some properties of the material occur differently in different directions. An example of the anisotropic energy for the current system is for $\beta_2,\beta_3\in \mathbb{R}$
    \begin{align*}
    	g(\phi)\equiv g((\phi^1,\phi^2,\phi^3))=(0,\beta_2 \phi^2,\beta_3 \phi^3)\, ,
    \end{align*}
    where we could read that two directions are getting a push in comparison to the first one, where no linear push is applied (of course other choices of the \enquote{privileged} axis are possible). The operator $\mathcal{H}(\phi)$ is minus the $L^2$ gradient of the energy functional, namely
    \begin{align*}
    	\mathcal{H} (\phi)=-\nabla\mathcal{E}(\phi)=\partial_x^2\phi- g'(\phi)\, ,
    \end{align*}
      for all $\phi\in L^2(D;\mathbb{S}^2)\cap H^1(D;\mathbb{R}^3)$.
     Landau and Lifschitz \cite{landau1935theory} provided for the first time a phenomenological derivation of the equation in 1935 and later, in 1954, Gilbert \cite{gilbert1954} considered and other damping parameter. The two equations proposed by Landau and Lifschitz and Gilbert have been proved to be mathematically equivalent by Brown \cite{brown1978}. See \cite{lakshmanan2011} for an overview on the physics. 
     
     Mathematically, the deterministic equation \eqref{LLG_intro} has been studied extensively. The existence and non uniqueness of weak solution on a two and three dimensional domain has been established by Alouges and Soyeur \cite{alouges1992global}. For the existence and uniqueness of strong solutions on a one dimensional domain, we refer to Y.~Zhu, B.~Guo, S.~Tan \cite{Zhu_Guo_Tan}. See also De Laire \cite{de_Laire} for a review on recent developments in the deterministic framework as well as  for the connection with other equations. In the deterministic setting, we mention existing results on steady states (i.e. stationary in time solutions): Carbou, Labbe \cite{Carbou_Labbe} for the one dimensional LLG and Van den Berg, J. Williams \cite{Carbou,van_den_Berg_Williams} for the two and three dimensional case.
     We stress here that, if $\lambda_1=0$, $\lambda_2=1$ and we do not consider the anisotropic energy,  \eqref{LLG_intro} coincides with the harmonic map flow on the sphere $\mathbb{S}^2\subset\mathbb{R}^3$ given by
     \begin{align*}
     	\partial_t u=\partial_x^2 u+u|\partial_x u|^2\, ,
     \end{align*}
     via the equality $-u_t\times (u_t\times \partial_x^2u_t)=\partial_x^2 u_t+u_t|\partial_x u_t|^2$ in $H^{-1}$ for $u_t(x)\in \mathbb{S}^2$ (see \cite[Lemma 2.4]{brzezniak_LDP}).
     Thus, the results of this paper apply also to the harmonic map flow.
     
     There are also several studies on the stochastic version of \eqref{LLG_intro}. Indeed the introduction of a noise in \eqref{LLG_intro}  models the fact that ferromagnetic materials are susceptible to temperature changes. As observed by Brown \cite{brown1963thermal, brown1978}, a way to do so is to include a Brownian fluctuation $W$ into the gradient of the energy: in this way the random changes of temperature are taken into account. Introducing a thermal fluctuation leads to a (not anymore phenomenological) description of the magnetization as a solution on $[0,T]\times D$ to the equation 
     \begin{equation}\label{LLG}
     \begin{aligned}
     \dd u =
     [\lambda_1  u\times \partial^2_x u
     -\lambda_2 u\times (u\times \partial_x^2 u)+\lambda_1 u\times g'(u)-\lambda_2u\times (u\times g'(u))]\dd t+u\times \circ dW\, ,
     \end{aligned}
     \end{equation}
     for an initial condition $u_0=u^0$, with null Neumann boundary conditions and a saturation condition on the material (which gives the mathematical spherical constraint)
     \begin{equation*}
     \partial_x u_t (x)=0\;\forall (t,x)\in [0,T]\times \partial D\, ,\quad |u_t(x)|=1\quad \forall (t,x)\in [0,T]\times D \, .
     \end{equation*}
    In this work the noise $W$ is chosen to be white in time and more regular in space and the integral is understood in the sense of Stratonovich. More precisely, throughout this introduction, we assume the following conditions on the noise.
    \begin{assumption}\label{assumption:noise_intro}
    	We consider noises either of the form $W_t(x):=h_1(x)B_t$, where $h_1\in H^1(D;\mathbb{R}^3)$, $B$ is a real valued Brownian motion, or of the form $W_t(x)=h_2(x)\bar{B}_t$, where $h_2\in H^1(D;\mathbb{R})$ and $\bar{B}$ is a $\mathbb{R}^3$-valued Brownian motion.
    \end{assumption}
     It is possible to consider linear combinations of the two noises in what follows, but we consider only one noise type at the time for clarity of exposition. The bounds in the theorems would change accordingly.
    
	The addition of the anisotropic energy as a linear operator into the equation does not constitute a mathematical difficulty when studying existence and uniqueness of the equation via our method (also in the stochastic case). On the other hand, as proved also in Brze\'zniak, Goldys and Jegaraj \cite{brzezniak_LDP}, the anisotropic energy is determinant to capture some physical behaviours of the magnetization, like the magnetization reversal phenomenon. The anisotropic energy gives an accurate description of the stray energy for one dimensional magnets. We include the anisotropic energy in the study. Throughout this introduction, we assume the following condition on the anisotropic energy.
	\begin{assumption}\label{assumption:anisotropic_intro}
		Assume that the anisotropy is of the form $g'(x):=Ax+b$, where $A$ is a real valued $3\times 3$ matrix and $b\in \mathbb{R}^3$.  
	\end{assumption}

	Concerning existence of weak solutions to the stochastic LLG in more dimensions, we mention the works \cite{brzezniak2013weak, hocquet2015landau, goldys2016finite,goldys2020weak} (but the list is not exhaustive).  
	Existence of a martingale solution $u$ to \eqref{LLG} on a one dimensional domain is established by  Brze\'zniak, Goldys, Jegaraj \cite{brzezniak_LDP} and Brze\'zniak, Manna, Mukherjee \cite{brzezniak2019wong}. The authors also prove pathwise uniqueness. Gussetti, Hocquet \cite{LLG1D} prove existence and uniqueness of a pathwise energy solution $u$ in $L^\infty(H^1)\cap L^2(H^2)\cap C([0,T],L^2)$ by means of rough paths theory.

	 The existence of a unique solution to \eqref{LLG} implies the existence of a Markov semigroup associated to $u$, that we denote by $(P_t)_t$. From a physical perspective, studying the semigroup $(P_t)_t$ corresponds to looking at the average behaviour of the trajectories at every time $t>0$. In other words, we observe where the trajectories are concentrated on the sphere at time $t$ or at the average configuration.
	
	 A typical question in this framework concerns the existence of a configuration on the sphere which does not vary in time, that we can refer to as stationary configuration. This translates into searching for the so-called invariant measures, which correspond to equilibrium states of the system (see e.g. \cite{corso_gibbs}). 
	The existence of  an invariant measure $\mu$ associated to $(P_t)_t$ on $H^1(\mathbb{S}^2):=L^2(D;\mathbb{S}^2)\cap H^1(D;\mathbb{R}^3)$ and the rigorous identification with the Gibbs measure
	\begin{align}\label{eq:gibbs_intro}
		\frac{\exp(-\mathcal{E}(y))\dd y}{\int_{H^1(\mathbb{S}^2)}\exp(-\mathcal{E}(v))\dd v}
	\end{align}
	for some of those invariant measures, are open problems. Note that the formulation in \eqref{eq:gibbs_intro} is only formal. Indeed the only existing translation invariant Borel measures on an infinite dimensional separable Banach space are either the null measure or the measure that assigns to every open set $+\infty$. As a consequence, there is no natural extension (in this sense) of a \enquote{Lebesgue}-like measure to an infinite dimensional separable Banach space. There are nevertheless different ways of constructing measures on an infinite dimensional space. A formal derivation of the fact that the invariant measure should be a Gibbs measure of the form \eqref{eq:gibbs_intro} is established by García-Palacios, Lázaro \cite{garcia_lazaro_1998}. 
	
	  In the framework of ferromagnetism, it is worth mentioning the Landau-Lifschitz-Bloch (LLB) equation, which is an alternative model to the LLG equation that holds below, close and above the Curie temperature. The LLB equation interpolates between the LLG at low temperatures and the Bloch equation, which holds for very high temperatures. In the LLB model, the magnetization is no more constrained to the sphere. The LLB equation has been introduced by Garanin \cite{garanin_1,garanin} and its stochastic version has been studied by Jiang, Ju, Wang \cite{Jang_Ju_Wang_LLB} and Brze\'zniak, Goldys and K.-N.~Le \cite{LLB_Brezniak}. The first named authors establish global existence of martingale weak solutions to the LLB equations. The second named authors establish existence and uniqueness of probabilistically strong martingale solutions to the equation and the existence of an invariant measure in one and two dimensions. 
	
	Despite the numerous studies on the stochastic LLG equation and related equations in ferromagnetism, the existence of invariant measures for the LLG has not been addressed for the full equation \eqref{LLG}.
	The only result approaching an approximated version of the problem is from Neklyudov, Prohl \cite{Neklyudov_Prohl}. The authors show rigorously the existence of a Gibbs invariant measure of a finite system of spins behaving like a simplified LLG equation: the construction occurs by means of a Lyapunov function. The authors also make crucial considerations on the role of the noise and its implications on invariant measures. This fact is explored in what follows. We refer to the monograph \cite{Banas_book} and the references therein for further details on the numerical approaches and the open problems concerning the LLG equation. 
	We mention the study

\paragraph{The main results}
	The aim of this paper is to prove existence of an ergodic invariant measures for the stochastic LLG equation on a one dimensional domain. We prove existence of an ergodic invariant measure in presence of both exchange energy and anisotropic energy (in the form of a bilinear operator). In some special cases, we can identify the invariant measure with the Gibbs distribution. We state the main results and the procedures.
	
	\paragraph{Existence of an ergodic invariant measure and stationary solutions.}
	\begin{theorem*}
	The semigroup $(P_t)_t$ on $H^1(\mathbb{S}^2)$ associated to \eqref{LLG} admits an invariant measure.
	\end{theorem*}
	
	We prove existence of an invariant measure (refer to Theorem \ref{teo:existence_invariant_measure}) by means of the classical Krylov-Bogoliubov theorem. The rough path formulation is useful when dealing with pathwise properties: here the deterministic calculus can be employed and one does not have to take in account integrability and measurability issues with respect to the probability measure. On the other hand classical stochastic calculus is of use when dealing with probabilistic properties, like the martingale property of the It\^o integral. 
	
	We use rough paths theory to show that the semigroup $(P_t)_t$ has the Feller property: the proof of this fact is a consequence of the continuity of the It\^o-Lyons map with respect to the initial condition and the noise, as already proved in \cite{LLG1D}. We include the proof and statements also in this paper for completeness of exposition. The Feller property is achieved pathwise in $H^1(\mathbb{S}^2)$ and no condition on the intensity of the noise is required. Rough paths theory is determinant to show the pathwise bound 
	\begin{align}\label{eq:continuity_intro}
	\sup_{t\in[0,T]}\|u_t-v_t\|^2_{H^1}\lesssim_{\exp(T),\mathbf{W}}\left[\|u^0\|^4_{H^1}+\|v^0\|^4_{H^1}\right] \|u^0-v^0\|^2_{H^1}\, ,
	\end{align}
	where $u,v$ are solutions to \eqref{LLG_intro} started at $u^0,v^0\in H^1(\mathbb{S}^2)$. The continuity with respect to the initial condition in $H^1(\mathbb{S}^2)$ implies the Feller property in Theorem \ref{teo:Feller}. In other words, the semigroup maps continuous bounded functions from $H^1(\mathbb{S}^2)$ to $\mathbb{R}$ to continuous bounded functions from $H^1(\mathbb{S}^2)$ to $\mathbb{R}$:
	\begin{theorem*}
	The semigroup $(P_t)_t$ has the Feller property, i.e. $P_t:C_b(H^1(\mathbb{S}^2))\rightarrow C_b(H^1(\mathbb{S}^2))$ for all $t>0$ .
	\end{theorem*}
    It is fundamental to stress that classical Stratonovich calculus does not allow to achieve easily \eqref{eq:continuity_intro}: indeed one needs to employ the It\^o's isometry to estimate the stochastic integral and therefore it is needed to take the expectation of the inequality. But the drift of the equation is multiplied by the initial conditions to the power four. As a consequence, by applying expectation, one is not able to close the estimate by means of the classical Gronwall's Lemma. We expand on this point in Remark \ref{remark:stratonovich_not_possible_feller}.

	The proof of tightness is achieved by the classical It\^o-Stratonovich calculus and, in presence of anisotropic energy, by means of Poincaré's inequality; we need to use that the It\^o integral is a martingale. Note that the Stratonovich solution of \eqref{LLG} introduced in \cite{brzezniak2019wong} and \cite{brzezniak_LDP} coincides with the pathwise solution in \cite{LLG1D}, a part from a set of null measure, it is possible to pass from one formulation to the other, under the hypothesis that $u^0\in \mathcal{L}^2(\Omega;H^1(\mathbb{S}^2))$. We formulate in a compact form the technical estimates necessary to conclude tightness (contained in Lemma \ref{lemma:energy_anisotropic}, Lemma \ref{lemma:linear_growth_tightness}).
	\begin{lemma*} There exists $C>0$ independent on time, so that for every $t>0$ it holds
		\begin{align*}
			\sup_{r\in[0,t]}\mathbb{E}\left[\|\partial_x u_r\|^2_{L^2}\right]+\int_{0}^{t}\mathbb{E}\left[\|u_r\times\partial^2_x u_r\|^2_{L^2}+\|\partial^2_x u_r\|^{1/2}_{L^2}\right]\dd r\lesssim_{\partial_xu^0,\partial_x h,g,\lambda_1,\lambda_2}C(1+t)\, .
		\end{align*}
	\end{lemma*}
   Note that it is possible to bound the norm $\|\partial_x^2u\|_{L^2}$ with respect to the expectation only with the power $1/2$. This bound requires to use the geometry of the equation.  
  
   Let $\mu$ be an invariant measure and consider an initial condition distributed like $\mu$. Then there exists a pathwise stationary solution $w$ with initial condition $w^0$. If the initial condition $u^ 0$ used in the Krylov-Bogolioubov theorem has bounded second moment, then each initial condition distributed like $\mu$ has bounded second moment. It is also possible to conclude that $\textrm{supp}\mu\subset H^2(\mathbb{S}^2)$. These results are contained in Theorem \ref{teo:pathwise_stationary_sol} and Theorem \ref{teo:regolarity_pathwise_stationary_sol}.  The condition $\mathbb{E}[\|w^0\|_{H^1(\mathbb{S}^2)}^2]<+\infty$ allows to show that the set of invariant measures is compact. Since the set of invariant measures is not empty and convex, from the Krein-Milman theorem the following result (contained in Theorem \ref{teo:beta_neq_0}) follows:
   \begin{theorem*}
   	The semigroup $(P_t)_t$ on $H^1(\mathbb{S}^2)$ associated to \eqref{LLG} admits an ergodic invariant measure.
   \end{theorem*}

\paragraph{Gibbs measures for spatially constant noise and small anisotropic energy.}
   Up to this point, the previous results do not require any restrictive assumption on the noise nor on the drift of equation \eqref{LLG_intro}. The shape of the noise does not play any role in the previous results, but it is important in what follows.
    We also employ the results in \cite{Neklyudov_Prohl} for some considerations on invariant measures of the LLG. We can interpret the Stratonovich integral in two basic ways 
   \begin{align*}
   	\int_{0}^{t} u_r\times h_1 \circ \dd B_r\, ,\quad\quad\quad\quad \int_{0}^{t} h_2 u_r\times  \circ \dd \bar{B}_r\, ,
   \end{align*}
	where $h_1\in H^1(D;\mathbb{R}^3)$, $B$ is a real valued Brownian motion, $h_2\in H^1(D;\mathbb{R})$ and $\bar{B}$ is a $\mathbb{R}^3$-valued Brownian motion. One could also consider linear combinations of the noises above. We can define precisely the shape of stationary solutions, by assuming $\partial_x h_1,\partial_x h_2=0$. Consider the equations 
	\begin{align*}
	\textrm{(A)}\quad\delta v_{s,t}=\mathcal{D}(v)_{s,t}+\int_{s}^{t}v_r\times h_1 \circ \dd B_r\, ,\quad \quad\quad \textrm{(B)}\quad\delta \bar{v}_{s,t}=\mathcal{D}(\bar{v})_{s,t}+\int_{s}^{t}h_2\bar{v}_r\times \circ \dd \bar{B}_r\, ,
	\end{align*}
	where $\mathcal{D}(\tilde{v})_{s,t}=\int_{s}^{t}[\lambda_2\tilde{v}_r\times(\tilde{v}_r\times A \tilde{v}_r)-\lambda_1\tilde{v}_r\times A \tilde{v}_r]\dd r$ for $\tilde{v}_r\in \mathbb{S}^2$ is the contribution coming from the anisotropic energy. We need a smallness assumption on the anisotropy.
	\begin{assumption}\label{assumption:anisotropic_intro_small}
		Assume that the anisotropy is of the form $g'(x):=Ax+b$, where $A$ is a real valued $3\times 3$ matrix and $b\in \mathbb{R}^3$.  Assume further that $\bar{G}:=2\sup_{i,j}|A_{i,j}|^2+|b|<\lambda_2/2C_p(2\lambda_2+|\lambda_1|)$, where $C_p$ is the Poincaré's constant associated to the domain $D$.
	\end{assumption}
	Refer to Theorem \ref{teo:sde_h_1} and Theorem \ref{teo:sde_h_2} for the next result.
	\begin{theorem*}
		Assume Assumption \ref{assumption:anisotropic_intro_small} and $\partial_x h_1=0 , h_1\neq 0\,\,(\textrm{resp. }\partial_x h_2=0, h_2\neq 0)$. Then every stationary solution to \eqref{LLG} is constant in space and it is a solution to $\textrm{(A)}$ (resp. $\textrm{(B)}$).		
	\end{theorem*}
	We follow the discussion by Neklyudov, Prohl \cite{Neklyudov_Prohl} on the choice of the noise for \eqref{LLG} and its implications. The deterministic LLG equation admits steady states (i.e. stationary in time solutions), which are in general not unique. Denote by $\bar{u}$ a steady state to the deterministic equation and set $h_1=\bar{u}$.	Then both $\bar{u}$ and $-\bar{u}$ are solutions to equation $\textrm{(A)}$ and, in general, we can not expect to have a unique stationary solution to \eqref{LLG}. For other examples, see Proposition 1.17 in \cite{Banas_book} (see also \cite{Neklyudov_Prohl}). 
	
	If we consider instead equation $\mathrm{(B)}$ and we assume $h_2\neq 0$, then the stationary in time solutions to the deterministic LLG are not solutions to equation $\mathrm{(B)}$. We also observe that equation $\mathrm{(B)}$ describes the so called \enquote{spherical Brownian motion} with an anisotropic drift. In particular, equation $\mathrm{(B)}$ has a unique solution for $h_2\neq 0$ and admits a unique invariant measure. The proof of the properties of  $\mathrm{(B)}$ is contained in Neklyudov, Prohl \cite{Neklyudov_Prohl} and applies to this setting. We state the main result, connecting the solution to equation $\mathrm{(B)}$ with the LLG equation. Refer to Theorem \ref{teo:sde_h_2}.
	\begin{theorem*}
	Assume Assumption \ref{assumption:anisotropic_intro_small} and $h_2\neq 0$, $\partial_x h_2=0$. Then there exists a unique ergodic invariant measure $\mu$ for \eqref{LLG}.  The unique invariant measure $\mu$ can be identified with a Gibbs measure on $(\mathbb{S}^2,\mathcal{B}_{\mathbb{S}^2})$, i.e.
	\begin{align}\label{eq:gibbs_sphere}
	\mu(	\dd y)=\frac{\exp(-\frac{\lambda_2}{h_2}\bar{\mathcal{E}}(y))\dd y}{\int_{\mathbb{S}^2}\exp(-\frac{\lambda_2}{h_2}\bar{\mathcal{E}}(v))\dd v}\, ,
	\end{align}
	where $\bar{\mathcal{E}}(v)=\int_{D}g'(v)\cdot v\dd x=|D|g'(v)\cdot v$,  for $v\in \mathbb{S}^2$.
	
	\end{theorem*}
	Note that the Gibbs measure above contains the anisotropic terms, but no derivatives. If there is no anisotropy, then the Gibbs measure reduces to the uniform measure on the sphere. 
	Observe also that the measure $\mu$ in \eqref{eq:gibbs_sphere} on $(\mathbb{S}^2,\mathcal{B}_{\mathbb{S}^2})$ is consistent with \eqref{eq:gibbs_intro}. Indeed consider the restriction of $ L^2(D;\mathbb{S}^2)\cap H^1(D;\mathbb{R}^3)$ to the space of constant maps $\phi_v:D\rightarrow \mathbb{S}^2$ so that $\phi_v(x):=v$ for all $x\in D$ and for all $v\in \mathbb{S}^2$. The space of constant maps $\{\phi_v\}_{v\in\mathbb{S}^2}$ can be identified with the sphere $\mathbb{S}^2$. Thus the energy \eqref{eq:energy_intro} evaluated in each $\phi_v$ coincides with
	\begin{align*}
	\mathcal{E}(\phi_v)=\int_{D}g(\phi_v)\cdot\phi_v \dd x=|D|g(\phi_v)\cdot\phi_v =|D|g(v)\cdot v=\bar{\mathcal{E}}(v)\, ,\quad\forall v\in\mathbb{S}^2\, .
	\end{align*}
	This shows the identification of the integration domain $H^1(\mathbb{S}^2)$ with $\mathbb{S}^2$ in this particular case, hence the correspondence of the formal Gibbs measure in \eqref{eq:gibbs_intro} with \eqref{eq:gibbs_sphere}.
	
	We can compare the stationary solutions to the LLG with the works of M. Röckner, B. Wu, R. Zhu, X. Zhu \cite{Rock_Zhu_1}, \cite{Rock_Zhu_2} on  stochastic heat equation taking values on a Riemannian manifold. The authors employ the theory of Dirichlet forms to prove existence of martingale solutions to the stochastic heat equation taking values on a Riemannian manifold. More specifically, they prove that the Wiener (Brownian bridge) measure is an  invariant measure on the Riemannian path loop space. In a simplified case, we are able to prove that the Brownian motion on the sphere is a stationary solutions to \eqref{LLG} and thus that the invariant measure associated to the LLG is the one of a Brownian motion on the sphere. This fact is in accordance to the results obtained in \cite{Rock_Zhu_1}, \cite{Rock_Zhu_2}.
	
	
	\paragraph{Long time behaviour for $\partial_x h_2=0, h_2\neq 0$ and $g\equiv 0$.}
	We establish also a time behaviour result under the condition $\partial_x h_2=0$: every solution converges to its spatial average, as a consequence of the Poincaré-Wirtinger theorem. In absence of anisotropic energy and provided $\partial_x h_2=0$, we observe that every solution to \eqref{LLG}, independently on the initial condition, is synchronizing with the stationary solution up to a constant (independent on space and time).  The result is contained in Proposition \ref{prop:long_time_behaviour} and modifies the proof for the stochastic mean curvature flow from Dabrock, Hofmanová, Rogers \cite{Dabrock_Hofmanova_Roger}.
	\begin{proposition*}
		Assume $h_2\neq 0$, $\partial_x h_2=0$ and no anisotropic energy, i.e. $g\equiv 0$. 
		For large times every solution $u$  to \eqref{LLG} converges to a Brownian motion $w$ with values on the sphere, where $w$ is the unique solution to 
		\begin{align*}
		\delta w_{s,t}=h_2\int_{s}^{t} w_r \times \circ \dd \bar{B}_r\, .
		\end{align*}
		The convergence occurs up to a constant, i.e. for $T\rightarrow +\infty$ for $t>0$
		\begin{align*}
		[u_{T+t}-w_{T+t}]^2-[u_{T}-w_T]^2=0
		\end{align*}
		in $C_b([0,+\infty);L^1)$ $\mathbb{P}$-a.s.
	\end{proposition*}
	We refer to Appendix \ref{sec:spherical_BM} for a rigorous definition of the spherical Brownian motion and for further references.
	Since the spherical Brownian motion is recurrent, then every solution to \eqref{LLG} satisfying the hypothesis of Proposition \ref{prop:long_time_behaviour} undergoes magnetization reversal for large times.
	\begin{corollary*}
		Assume $h_2\neq 0$, $\partial_x h_2=0$ and no anisotropic energy, i.e. $g\equiv 0$.  Every solution to \eqref{LLG} undergoes magnetization reversal for large times.
	\end{corollary*}
We conclude with several questions left open, namely the identification of an ergodic invariant measure with the Gibbs measure if $\partial_x h_2$ is not null and the number of invariant measures in that case. It would also be interesting to study the long time behaviour for \eqref{LLG} if $\partial_x h_2\neq 0$ and in presence of anisotropic energy. 

\paragraph{Organisation of the paper:} In Section~\ref{sec:notation} we introduce some basic notations and the rough driver. In Section~\ref{sec:def_LLG}, we introduce the definition of solution to \eqref{LLG} and we recall some known results. In Section \ref{sec:inv_measure}, we recall the basics concerning invariant measures and the Krylov-Bogoliubov theorem: in Section~\ref{sec:strong_feller} we show that the Feller property holds and in Section~\ref{sec:tightness} we prove the tightness of the sequence of measures in the Krylov-Bogoliubov theorem. In Section~\ref{sec:ergodic_inv_measure}, we employ the Krein-Milman theorem to achieve the existence of an ergodic measure. We identify the invariant measure with the unique Gibbs measure in absence of anisotropy and if $\partial_x h_2=0$, $h_2\neq 0$ in Section~\ref{sec:uniq_stat_sol_inv_measure}. In presence of anisotropy and provided  $\partial_x h_2=0$, $h_2\neq 0$, we determine the invariant measure to be the Gibbs measure in Section~\ref{sec:uniq_stat_sol_inv_measure}. In Section~\ref{sec:large_time_behaviour}, we address the long time behaviour problem by showing the convergence of the solutions to the LLG to the stationary solution, under the condition $\partial_x h_2=0$, $h_2\neq 0$ and in absence of anisotropic energy.

\paragraph{Acknowledgements:} The author is warmly grateful to Professor M.~Hofmanov\'a for her many advices and constant help to run this project and for pointing out the procedure for the long time behaviour in \cite{Dabrock_Hofmanova_Roger}. The author is also thankful to Professor B. Goldys many interesting discussions in the last phase of this work and for pointing out Remark \ref{Goldys}, the sensitivity of the equation to the choice of the noise and for pointing out several references \cite{Rock_Zhu_1,Rock_Zhu_2}. The author thankfully acknowledges the financial support by the German Science Foundation DFG via the Collaborative Research Center SFB1283, Project B7. The author is also thankful to the organizers of the GDR-TRAG Young researchers meeting in Paris in December 2021. The author wishes to thank the anonymous reviewer for the useful remarks.

\section{Notations and setting}\label{sec:notation}
\subsection{Frequently used notations.}
For $a,b\in \R^3$, we denote by $a\cdot b$ the inner product in $\R^3$, and by $|\cdot|$ the norm inherited from it (we will not distinguish between the different dimensions, it will be clear from the context). We recall the definition of cross product
$a\times b:=(a_2b_3-a_3b_2,a_3b_1-a_1b_3,a_1b_2-a_2b_1)$, for $a\equiv(a_1,a_2,a_3),b\equiv(b_1,b_2,b_3)\in \R^3$.
We denote by $\mathbb{S}^2:=\{a\in\mathbb{R}^3:|a|_{\mathbb{R}^3}=1 \}$ the unit sphere in $\R ^3$.
 The space of bounded linear operators from a Banach space $E$ to itself is denoted by $\mathscr L(E)$ and by $\mathbf 1_{3\times 3}$ the identity operator in $\mathscr{L}(\mathbb{R}^3)$. The space $\mathscr{L}_a(\mathbb{R}^3)$ is the subspace of $\mathscr{L}(\mathbb{R}^3)$ of antisymmetric matrices, i.e. $\Gamma\in \mathscr{L}(\mathbb{R}^3)$ such that $\Gamma v\cdot w = -v\cdot(\Gamma w)$ $\forall v,w\in\R^3$. Given $a,b\in\mathbb{R}^3$, we denote by $a\otimes b$ the tensor product in $\mathbb{R}^3$, i.e. $a\otimes b=(a_ib_j)_{i,j=1,2,3}$.

\paragraph{Paths and controls}
Let $J\subset [0,T]$ be a subinterval and $J^2:=J\times J$. For convenience of notations, we will write `$\forall s \leq t \in J$'  instead of `$\forall (s,t)\in J^2$ such that $s \leq t.$'
For a one-index map $h$ defined on $J$, we use the notation $\delta h_{s,t}:=h_t-h_s$ for $s \leq t \in J$.
For a two-index map $H$ defined on $J^2$, we define $\delta H_{s,u,t}:=H_{s,t}-H_{s,u}-H_{u,t}$ for $s \leq u\leq t\in J$. 
We call \textit{increment} any two-index map which is given by $\delta h_{s,t}$ for some $h=h_t$.
Observe that increments are exactly those two-index elements $H$ for which $\delta H_{s,u,t}\equiv 0$.

We say that a continuous map $\omega\colon \{(s,t)\in J:s\leq t\}\rightarrow [0,+\infty)$ is a \textit{control on $J$} if $\omega$ is continuous, $\omega(t,t)=0$ for any $t\in J$ and if it is \textit{super-additive}, namely for all $s \leq u \leq t$
\begin{align*}
\omega(s,u)+\omega(u,t)\leq \omega (s,t)\, .
\end{align*}
Given a control $\omega$ on $J:=[s,t]$, we also denote $\omega(J):=\omega(s,t)$.

Let $(E,\|\cdot\|_E)$ be a Banach space and $p>0$, we denote by $\mathcal{V}^p_{2}(J;E)$ the set of two-index maps $H\colon\{(s,t)\in J^2 :s\leq t\}\rightarrow E$ that are continuous in both the components such that $H_{t,t}=0$ for all $t\in J$ and there exists a control $\omega$ on $J$ such that
\begin{equation} \label{condition_control_1}
\|H_{s,t}\|_E\leq \omega(s,t)^{1/p}
\end{equation}
for all $s\leq t \in J$. 
The space $\mathcal{V}^p_{2}(J;E)$ is equivalently defined as the space of two-index maps of finite $p$-variation, namely $H$ belongs to $\mathcal{V}^p_{2}(J;E)$ if and only if 
\[
\|H\|_{\mathcal{V}_{2}^{p}(J;E)}:=\sup_{\mathcal{P}}\Big (\sum\nolimits_{[s,t]\in\pi}\|H_{s,t}\|_E^p\Big)^\frac{1}{p}<+\infty\, ,
\]
where the supremum is taken over the set of partitions $\mathcal{P}=\{[s_0,s_1],\dots [s_{n-1},s_{n}]\}$ of $J$.
Moreover, the semi-norm $\|\cdot\|_{\mathcal{V}^p_{2}(J;E)}$ coincides with the infimum of $\omega(J)^p$ over the set of controls $\omega$ such that the condition \eqref{condition_control_1} holds (see \cite{hocquet2017energy} or \cite[Paragraph 8.1.1]{FrizVictoire}).	
Analogously we can define $\mathcal{V}^p(J;E)$ as the space of all continuous paths $g\colon J\to E$ such that $\delta h\in\mathcal{V}^p_{2}(J;E)$, equipped with the norm $\|h\|_{\V^p(J;E)}=\sup_{t\in J}\|h_t\|_{E}+\|\delta h\|_{\mathcal{V}_{2}^{p}(J;E)}$.

We will sometimes need to work with local versions of the previous spaces: we define $\mathcal{V}_{2,\mathrm{loc}}^{p}(J;E)$ as the space of maps $H\colon\{(s,t)\in J^2:s\leq t\}\rightarrow E$ such that there exists a finite covering $(J_k)_{k\in K}$ of $J$, $K\subset \mathbb{N}$, so that $H\in \mathcal V_{2}^{p}(J_k;E)$ for all $k\in K$. 
We define the linear space 
\begin{align*}
\mathcal{V}^{1-}_{2,\mathrm{loc}}(J;E):=\bigcup_{0<p<1} \mathcal{V}_{2,\mathrm{loc}}^{p}(J;E)\, .
\end{align*}
Fix $T>0$. We denote by $\mathcal{V}^p(E)=\mathcal{V}^p([0,T];E),$. Analogously, we use the short hand \( \mathcal{V}_2^p(E)=\mathcal{V}_2^p([0,T];E) \). We denote the space of continuous functions defined on $[0,T]$ with values in $E$ by $C([0,T];E)$ . The space $C^k_b([0,+\infty); \mathbb{R})$ indicates the space of bounded and continuously differentiable functions with bounded derivatives, defined on $[0,+\infty)$ with real values, for $k\in \mathbb{N}$. The notation $\dot{h}$ denotes the time derivative of $h\in C^k_b([0,+\infty); \mathbb{R})$.
\paragraph{Sobolev spaces}
Let $D\subset\mathbb{R}$ be an open bounded interval of $\mathbb{R}$. Denote by $\mathbb{N}$ the space of natural numbers and $\mathbb{N}_0:=\mathbb{N}\cup\{0\}$.
For $n\in\mathbb{N}$, we consider the usual Lebesgue spaces $L^p:=L^p(D ;\R^n)$, for $p\in[1,+\infty]$ endowed with the norm $\|\cdot\|_{L^p}$ and the classical Sobolev spaces $W^{k,q}:=W^{k,q}(D;\R ^n)$ for integer $q\in [1,+\infty]$ and $k\in\mathbb{N}$ endowed with the norm $\|\cdot\|_{W^{k,q}}$.
We also denote by $H^k:=W^{k,2}(D ;\R ^n)$. 
We need to consider also functions taking values in $\mathbb{S}^2\subset\R^3$: we therefore introduce the notation
\begin{equation}
H^k(\mathbb{S}^2):=H^k(D;\R^3)\cap \{g:D \rightarrow \R^3 \, \textrm{s.t.}\, |g(x)|=1\,\text{ a.e.}\, x\in D \}\, ,
\end{equation}
for $k\in\mathbb{N}_0$. 
Finally, we will denote by $L^p(W^{k,q}):=L^p([0,T];W^{k,q}(D ;\R ^n))$. We indicate with $C^k_0(D)$ the space of real valued functions with compact support on $D$, $k$-times continuously differentiable and such that every derivative is compactly supported on $D$. Let $(\Omega,\mathcal{F},\mathbb{P})$ be a probability space. We denote by $\mathcal{L}^p(\Omega;E)$ the usual Lebesgue space with respect to the probability measure $\mathbb{P}$.
\subsection{Construction of the rough driver.}
We briefly introduce the noise, as constructed  in \cite{LLG1D}. We are interested in interpreting the Stratonovich integral
\begin{align}\label{eq:integral}
	\int_{s}^{t} h u_r\times \circ \dd W_r\, ,
\end{align}
by means of rough path theory. We do not employ the classical definition of rough path, but we use the notion of rough driver (see \cite{bailleul2017unbounded}, \cite{deya2016priori,hocquet2017energy,hocquet2018ito,hofmanova2019navier,hocquet2018quasilinear}), that we introduce in Definition \ref{defi-RD}.
\begin{definition}[Rough driver]
	\label{defi-RD}
	Let $n\in\mathbb{N}$ and $p\in[2,3)$. A pair
	\begin{equation}\label{p-var-rp}
	\mathbf{B} :=(B,\mathbb{B} ) \in \mathcal V^p_2 \left ([0,T];L^2(D ;\mathscr L(\R^n))\right ) \times \mathcal V^{p/2}_2 \left ([0,T]; L^2(D ;\mathscr L(\R^n))\right ) 
	\end{equation}
	is said to be a $n$-dimensional \textit{rough driver} provided $\delta B_{s,u,t}=0$ and the relation is fulfilled as the sense of composition of linear maps in $\R^n$:
	\begin{equation}\label{chen-rela}
	\delta \mathbb{B} _{s,u,t}(x)=B_{u,t}(x)B_{s,u}(x)\,,
	\quad \text{for all }s<u<t\in [0,T],
	\quad \text{Lebesgue-a.e.\ }x\in D\,.
	\end{equation}
	We refer to \eqref{chen-rela} as Chen's relation. We introduce other properties of a rough driver.
	\begin{enumerate}[label=(\alph*)]
		\item
		$\mathbf{B}$ is \textit{geometric} if it can be obtained as the limit with respect to the $(p,p/2)$-variation topology, of a sequence of smooth rough drivers $\mathbf{B} ^\epsilon=(B^{\epsilon},\mathbb
		B ^{\epsilon}),$ where $\epsilon>0$, explicitly defined for $s\le t \in [0,T]$ as
		\begin{equation}\label{relazioni_path}
		B^{\epsilon}_{s,t} := \Gamma^{\epsilon}_{t}-\Gamma^\epsilon_s \ , 
		\qquad \mathbb{B} ^{\epsilon}_{s,t}:=\int_{s}^t\dd \Gamma^\epsilon_r (\Gamma^\epsilon_{r}-\Gamma^\epsilon_s)\, ,
		\end{equation}
		for some smooth path $\Gamma^\epsilon\colon[0,T] \to L^2(D;\mathscr L(\R^n))$.
		
		\item \label{itm:anti}
		$\mathbf{B} $ is \textit{anti-symmetric} if it is geometric and such that the approximating sequence $(\Gamma_t^\epsilon(x))$ in \eqref{relazioni_path} can be taken with values
		in the space $\mathscr L_a(\R^n)\subset\mathscr L(\R^n)$ (this implies in particular that the first component $B_{s,t}(x)$ takes values in \( \mathscr L_a(\R^n) \)). 
	\end{enumerate}
Given a rough driver $\mathbf{B}=(B,\mathbb{B})$, we call $B$ \textit{first iterated integral} and $\mathbb{B}$ \textit{second iterated integral}.
\end{definition}

\begin{remark}
An equivalent definition of geometric rough driver is that  $\textrm{Sym}(\mathbb{B}_{s,t})=(\mathbb{B}_{s,t}+\mathbb{B}_{s,t}^T)/2=B_{s,t}B_{s,t}/2$. This corresponds to the usual rough paths framework (\cite{FrizHairer}) in the following way: the main difference is that the first iterated integral of a rough path is normally a vector, thus the matrix product $B_{s,t}B_{s,t}$ is replaced with a cross product (see e.g. \cite{FrizHairer}). The geometricity in (a) property is thus directly inherited from the rough path setting and adapted to the fact that the first iterated integral of the rough driver is a matrix.

\end{remark}

We construct a rough driver so that the first iterated integral in \eqref{eq:integral} takes value on the sphere $\mathbb{S}^2$: this is done by using the structure of the cross product. Let $(\Omega,\mathcal{F},(\mathcal{F}_t)_t,\mathbb{P})$ be a filtered probability space. Consider a Brownian motion $\textit{w}\equiv(\textit{w}^1,\textit{w}^2,\textit{w}^3):\Omega\times [0,T]\rightarrow \mathbb{R}^3$ on $(\Omega,\mathcal{F},(\mathcal{F}_t)_t,\mathbb{P})$ and a map $\mathscr{F}:\mathbb{R}^3\rightarrow \mathscr{L}_a(\mathbb{R}^3)$ defined by $\mathscr{F}(\xi):=\cdot \times \xi$. Let $\Omega^{Str}\subset\Omega$ be the set of full measure such that  $w^{i,,j}_{s,t}(\omega)=\int_{s}^{t}\textit{w}^{i}_{s,r}\circ \dd \textit{w}^j_r(\omega)$ exists for all $\omega\in \Omega^{Str} $ and for all $i,j=1,2,3$, i.e. the set of $\omega\in \Omega$ such that the Stratonovich integral of $\textit{w}$ against itself exists.

For all fixed $\omega\in \Omega^{Str}$, we define the first iterated integral of the rough driver $\mathbf{W}(\omega)=(W(\omega),\mathbb{W}(\omega))$  as 
\begin{equation}
W_{s,t}(\omega):= \mathscr F(\delta\textit{w}_{s,t}(\omega)) 
=
\begin{pmatrix}
0 & \delta\textit{w}^{3}_{s,t}(\omega) & -\delta\textit{w}^{2}_{s,t}(\omega)\\
-\delta\textit{w}^{3}_{s,t} (\omega)& 0 & \delta\textit{w}^{1}_{s,t}(\omega)\\
\delta\textit{w}^{2}_{s,t} (\omega)& -\delta\textit{w}^{1}_{s,t}(\omega) & 0
\end{pmatrix}\, ,
\end{equation}
and the second iterated integral as
\begin{equation}
\mathbb{W}_{s,t}(\omega)
=
\begin{pmatrix}
-w^{3,3}_{s,t}(\omega) -w^{2,2}_{s,t} (\omega)&
w^{1,2}_{s,t}(\omega) &
w^{1,3}_{s,t}(\omega)\\
w^{2,1}_{s,t}(\omega) &
-w^{3,3}_{s,t}(\omega) -w^{1,1}_{s,t} (\omega)&
w^{2,3}_{s,t}(\omega)\\
w^{3,1}_{s,t} (\omega)&
w^{3,2}_{s,t} (\omega)&
-w^{2,2}_{s,t} (\omega)-w^{1,1}_{s,t}(\omega)\\
\end{pmatrix}\, .
\end{equation}
The couple $\textbf{W}\equiv(W,\mathbb{W})$ is a three dimensional random anti-symmetric geometric rough driver, i.e. for all $\omega \in \Omega^{str}$ the couple $\textbf{W}(\omega)$ is a three dimensional anti-symmetric geometric rough driver. A way to add spatial dependency to the previous example is to let, for some $h\in H^k(D;\R)$,
\begin{equation}
\label{simple_RD_h}
\mathbf{W}_h\equiv\left(h(x)W, h(x)^2 \mathbb{W}\right)\, .
\end{equation}
We denote by $\textit{RD} ^p_a(H^k)$, for $k\in\mathbb{N}$, the space of  $3$-dimensional anti-symmetric rough drivers such that for $h\in H^k(D;\mathbb{R})$
\begin{align*}
\mathbf
W_h \in \mathcal V^p_2 \left (0,T;H^k(D;\mathscr L_a(\R^3))\right ) \times \mathcal V^{p/2}_2 \left (0,T; H^k(D;\mathscr L(\R^3))\right ) \, ,
\end{align*}
whose coordinates belong to the $k$-th order Sobolev space $H^k.$
We introduce also the following controls for $s\leq t$ and $0\leq \gamma\leq k$
\begin{align*}
&\omega_{W,H^\gamma}(s,t):=\|W\|^p_{V^p_2 \left (s,t;H^\gamma(D;\mathscr L_a(\R^3))\right ) }, \quad \omega_{\mathbb
	W ,H^\gamma}(s,t):=\|\mathbb
W \|^{p/2}_{V^{p/2}_2 \left (s,t;H^\gamma(D;\mathscr L(\R^3))\right ) }\, ,\\
&\omega_{\mathbf{W} ,H^\gamma}(s,t):= \|W\|^p_{V^p_2 \left (s,t;H^\gamma(D;\mathscr L_a(\R^3))\right ) }+\|\mathbb
W\|^{p/2}_{V^{p/2}_2 \left (s,t;H^\gamma(D;\mathscr L(\R^3))\right ) }\, .
\end{align*}
We could consider rough drivers that are more general than the construction in \eqref{simple_RD_h}, but we restrict our investigation to rough drivers constructed as above by a three dimensional Brownian motion: indeed we rely on the martingale properties of the stochastic It\^o integral.
\subsection{Other constructions of the stochastic integral.}
We mention an other possible construction of the noise in \eqref{LLG_intro}. Let $h\in H^k(D;\mathbb{R}^3)$ and the Stratonovich lift to a rough path of a real valued Brownian motion $\mathbf{W}=(W,\mathbb{W})$. 
We can interpret the stochastic integral in the classical rough path setting as
\begin{align*}
\int_{s}^{t} u_r\times h\circ \dd W_r= W_{s,t}u_s\times h +\mathbb{W}_{s,t}u_s\times h+(u\times h)^{\natural}_{s,t}\, .
\end{align*}
The same integral can be interpreted in the Stratonovich sense. One can consider also linear combination of different Brownian motions, by using the formalism above. For what concerns the proof of existence of invariant measures, the choice of the noise has no consequences. On the other side, the choice of the noise affects the number of invariant measures. We return to this point in Section \ref{sec:uniq_stat_sol_inv_measure}. 

\section{The stochastic Landau-Lifschitz-Gilbert equation in one dimension }\label{sec:def_LLG}
We introduce the notion of solution of stochastic LLG that we employ in what follows.
	\begin{definition}
		\label{def:solution}
		Consider a Brownian motion $\textit{w}$ with values in $\mathbb{R}^3$ defined on a probability space $(\Omega,\mathcal{F},\mathbb{P})$. We consider the filtered probability space $(\Omega,\mathcal{F},(\mathcal{F}_t)_t,\mathbb{P})$ endowed with the complete natural filtration induced by $\textit{w}$.
		Let $\omega\mapsto \mathbf{W}_h(\omega)\equiv(hW(\omega),h^2\mathbb{W}(\omega))$ be a pathwise random rough driver belonging to $ \textit{RD} ^p_a(H^k)$ for some $ p\in [2,3) $, constructed from $\textit{w}$ as in \eqref{simple_RD_h}. Let $D\subset \mathbb{R}$ an open bounded interval.
		We say that a stochastic process $u\colon \Omega\times [0,T]\to L^2(D;\R^3)$  is a \textit{pathwise solution} of \eqref{LLG} if it fulfils
		\begin{enumerate}[label=(\roman*)]
			\item \label{sol_i}
			$u_t(\omega,x)\in\mathbb{S}^2$ for a.e.\ $(\omega,t,x)\in\Omega\times[0,T]\times D$;
			\item \label{sol_ii} 
			$u\in L^\infty (0,T;H^1)\cap L^2(0,T;H^2)$ $\mathbb{P}$-a.s. ; 
			\item \label{sol_iii}
			there exists $q<1 $ and a random variable $u^{\natural}\in \mathcal{L}^0(\Omega;\V^{q}_{2,\mathrm{loc}}(0,T;L^2))$ such that
			\begin{equation}
			\label{rLLG_def}
			\begin{aligned}
			\delta u_{s,t}&+\int_s^t\left[\lambda_2u_r\times (u_r\times\partial_x^2 u_r) -\lambda_1 u_r\times\partial_x^2 u_r\right]  \dd r\\
			&+\int_s^t\left[\lambda_2 u_r\times (u_r\times g'(u_r))-\lambda_1 u_r\times g'( u_r)\right]  \dd r
			&=hW_{s,t}u_s +h^2 \mathbb W_{s,t}u_s + u^\natural_{s,t}\,,
			\end{aligned}
			\end{equation}
			as an equality in $L^2(D ;\R^3)$, for every $s\leq t\in[0,T] $ and $\mathbb P$-a.s.
			\item (initial condition) $u_0=u^0$, where $u^0\in H^1(\mathbb{S}^2)$ $\mathbb P$-a.s. 
			\item (boundary conditions) $\partial_x u_t(x)=0$ for all $(t,x)\in \partial D\times [0,T]$ $\mathbb P$-a.s.
		\end{enumerate}
	\end{definition}
Observe that, if $u\in L^\infty (0,T;H^1)\cap L^2(0,T;H^2)$, then the equality
 $-u_t\times (u_t\times \partial_x^2u_t)=\partial_x^2 u_t+u_t|\partial_x u_t|^2$ in $H^{-1}$ for $u_t(x)\in \mathbb{S}^2$ (see \cite[Lemma 2.4]{brzezniak_LDP}) holds in $L^2$ for a.e. $t\in [0,T]$: indeed, since $u\in L^\infty(0,T;L^\infty)$ and $\partial_x^2 u\in L^2(0,T;L^2)$, then $u\times (u\times \partial_x^2u)\in L^2(0,T;L^2)$.
 Moreover, $u|\partial_x u|^2\in L^2(0,T;L^2)$ from the spherical bound and from the interpolation inequality in Lemma \ref{lemma:interp_ladyz}.
 
In \cite{LLG1D}, it is proved that there exists a unique solution to \eqref{LLG} in the sense of Definition \ref{def:solution}, provided $g\equiv 0$. Again under the assumption $g\equiv 0$, the solution is also continuous with respect to the noise and to the initial data in $H^1(\mathbb{S}^2)$ with respect to the norm $L^\infty (0,T;H^1)\cap L^2(0,T;H^2)\cap \mathcal{V}^p(W^{2,+\infty})$ (see Theorem 5.5 in \cite{LLG1D}). The continuity with respect to the initial data suggests that the semigroup associated to the solution $u$ has the Feller property and that we can try to employ the Krylov-Bogoliubov Theorem to show the existence of an invariant measure.
We remark that the solution in Definition \ref{def:solution} coincides with the solution of \eqref{LLG} where the noise is interpreted with respect to the classical Stratonovich calculus (as studied in \cite{brzezniak2019wong, brzezniak_LDP}): we can therefore pass from one formulation to the other. In \cite{LLG1D}, the solution to \eqref{LLG} is studied on the one-dimensional torus $\mathbb{T}$ and in absence of anisotropic energy. 

We are interested in dealing with non null anisotropic energy, i.e. $g\neq 0$.
The addition of anisotropic energy does not imply any significant change in all the results in \cite{LLG1D}: indeed the solution $u\in L^\infty(D;\mathbb{R}^3)$ and the anisotropy is trivially in $L^\infty (L^2)$.

The change of domain from the torus to a domain with boundary requires a slightly modified product formula, as we discuss in Section \ref{sec:strong_feller}. The proofs present minor changes, for instance when the divergence theorem is involved, but also in this case the null Neumann boundary conditions makes the change not noticeable. We refer to the proof of Lemma \ref{lemma:feller_property}, where we apply these changes to the continuity with respect to the initial datum in $H^1(\mathbb{S}^2)$. The changes regarding the proof of existence and the proof of uniqueness are similar as in the the proof of Lemma \ref{lemma:feller_property}.

\section{Existence of an invariant measure}\label{sec:inv_measure} 
We introduce some definitions and basic results regarding invariant measures associated to Markov processes. 
 Let $E$ be a Polish space endowed with the Borel $\sigma$-algebra $\mathcal{B}_E$.
 Let $B_b(E)$ be the space of bounded Borel measurable functions defined on $E$ with values in $\mathbb{R}$ and let $C_b(E)$ be the space of continuous bounded functions from $E$ with values in $\mathbb{R}$ .
Let $(P_t)_t$ be a Markov semigroup defined on $B_b(E)$. If $(P_t)_t$ maps $C_b(E)$ into $C_b(E)$, we say that $(P_t)_t$ has the \textit{Feller property}. Denote by $\mathcal{P}(E)$  be the space of probability measures on $(E,\mathcal{B}_{E})$.
The semigroup $(P_t)_t$ induces a dual semigroup $(P_t^*)_t$ on $\mathcal{P}(E)$, defined for all $\nu\in \mathcal{P}(E)$ and $\phi\in C_b(E)$ by
\begin{align*}
\int_{E}\phi d(P^*_t \nu)=\int_{E}P_t\phi d\nu \, .
\end{align*}
We say that a measure $\mu\in \mathcal{P}(E)$ is \textit{invariant} for the semigroup $(P_t)_t$ if $P^*_t\mu=\mu$ for all $t>0$.
The Krylov-Bogoliubov theorem furnishes a way to prove constructively the existence of an invariant measure (see e.g. \cite{Book_M}). 
\begin{theorem}\label{teo:Krylov_B}
	Let $E$ be a Polish space and let $(P_t)_t$ be a Markov semigroup with the Feller property on $C_b(E)$. Consider a random variable $u^0$ with values in $E$ with law $\mu$ and denote by $\mu^{u^0}_{t}:=P_t^*\mu$.
	
	Assume that there exists a divergent monotone increasing sequence of times $(t_n)_n$ so that the sequence of probability measures $(\mu_{t_n})_n\subset \mathcal{P}(E)$ defined for all $A\in \mathcal{B}_{E}$ by
	\begin{align*}
		\mu_{t_n}(A):=\frac{1}{t_n}\int_{0}^{t_n}\mu^{u^0}_s(A)  \dd s
	\end{align*}
	is tight. Then there exists at least one invariant measure for $(P_t)_t$. 
	
\end{theorem}

Let $u^0$ be a random variable with values in $ H^1(\mathbb{S}^2)$ and independent of the Brownian motion $W$. Consider $u=\pi(u^0,W)$ to be the unique solution to \eqref{LLG}, then the stochastic process $(u_t^{u^0})_{t}$ is Markov process with respect to the filtration generated by $W$ and $u^0$.
In our context, we are interested in studying the existence of an invariant measure on $H^1(\mathbb{S}^2)$ for the Markov semigroup of linear operators $(P_t)_t$ defined for all $\phi\in B_b(H^1(\mathbb{S}^2))$ by
\begin{align}\label{eq:transition_semigroup}
P_t\phi (x):=\mathbb{E}\left[\phi(u^x_t)\right]\, ,
\end{align}
where $u^x_t$ is the solution to \eqref{LLG} at time $t$ with initial condition $x\in H^1(\mathbb{S}^2)$. As stated in the Krylov-Bogoliubov theorem, we need to show that the semigroup $(P_t)_t$ has the Feller property and the tightness property of $(\mu_{t_n})_{n}$ in $H^1(\mathbb{S}^2)$, where we can rewrite $\mu_{t_n}$ as
\begin{align}\label{eq:mu_T}
\mu_{t_n}(\,\cdot\,)=\frac{1}{t_n}\int_{0}^{t_n}\mathbb{P}(u^{x}_s\in\cdot\,)  \dd s\, .
\end{align}
In Section \ref{sec:strong_feller}, we prove that $(P_t)_t$ has the Feller property by means of a rough path approach. The inequality which leads to the Feller property of the semigroup is already contained in \cite{LLG1D} for the solution to \eqref{LLG} defined on the one dimensional torus and in absence of anisotropic energy. We recall the proof for the reader's convenience by adding the anisotropic energy (which does not constitute a mathematical problem) and the boundary condition (which requires a slightly modified product formula, in Proposition \ref{pro:product}). In Section \ref{sec:tightness}, we prove that the sequence $(\mu_{t_n})_n$ is tight in $H^1(\mathbb{S}^2)$: this fact follows from some geometric properties of the equation, a proper a posteriori estimate of solution to the equation and the Markov property.

\begin{remark} 
	Note that the semigroup $(P_t)_t$ exists. Indeed the unique solution $u$ on $[0,T]$ is adapted to $(\mathcal{F}_t)_t$ (it is a continuous function of the rough path, as a consequence of the Wong-Zakai result in \cite{LLG1D}). The solution is also continuous in time with values in $L^2$. Since $\phi\in B_b(H^1(\mathbb{S}^2))$ is $\mathcal{F}$-measurable, also the composition $\phi(u)$ is $\mathcal{F}$-measurable. We also know that $\phi$ is bounded, therefore $P_t\phi(x)$ is well defined for every time $t\in [0,T]$.

For every initial condition $x\in H^1(\mathbb{S}^2)$ there exists a unique solution $u=\pi(u^0,W)$ on every fixed time interval $[0,T]$. As a consequence of uniqueness, for  every $s\leq t\in [0,T]$, it holds that for all $\phi\in C_b(H^1(\mathbb{S}^2))$
\begin{align*}
\mathbb{E}[u^x_{s+t}|\mathcal{F}_s]=\mathbb{E}[\phi(u^x_{s+t})]=P_t\phi(u_s^x)\, ,\quad\quad \mathbb{P}-\textrm{a.s.}\, ,
\end{align*}
which shows that $(P_t)_t$ is a Markov semigroup on $H^1(\mathbb{S}^2)$ with respect to $(\mathcal{F}_t)_t$. We do not specify the time interval, which always coincides with the interval of existence of the solution to \eqref{LLG}.
\end{remark}

\subsection{Feller property in the $H^1(\mathbb{S}^2)$-norm via rough paths.} \label{sec:strong_feller}
In order achieve the Feller property of the semigroup $(P_t)_t$, we need a small extension of the continuity with respect to the initial datum contained in \cite[Theorem 5.5, equation (5.13)]{LLG1D}, which exploits the rough path formulation of the equation: for fixed $T>0$ we know that the pathwise solution depends continuously on the initial datum, namely given two initial conditions $u^0,v^0\in H^1(\mathbb{S}^2)$ and the corresponding solutions $u=\pi(u^0,\mathbf{W})$ and $v=\pi(v^0,\mathbf{W})$ to \eqref{LLG}, the following inequality holds 
	 \begin{align}\label{eq:cont_init_data_1}
	 \sup_{t\in[0,T]}\|u_t-v_t\|_{H^1}\leq C \exp\left([\|u^0\|^2_{H^1}+\|v_0\|^2_{H^1}]t+\omega_{\mathbf{W}}^{1/p}\right) \|u^0-v^0\|_{H^1}\, ,
	 \end{align}
	 $\mathbb{P}-$a.s. and where the constant $C>0$ is independent on time and on the initial condition. We note that the constant depends exponentially on time and exponentially on the initial conditions. More specifically, if we consider two initial conditions $u^0,v^0\in B_R(0)\subset H^1(\mathbb{S}^2)$ the ball in $H^1(\mathbb{S}^2)$ of radius $R>0$ centred in $0$, then the constant in \eqref{eq:cont_init_data_1} depends on the initial condition only as a positive power of $R$.
	 
	  The pathwise convergence in \eqref{eq:cont_init_data_1} implies the Feller property, as stated in Theorem \ref{teo:Feller}.
	  \begin{theorem}\label{teo:Feller}
	  	The semigroup $(P_t)_t$ has the Feller property, i.e. $P_t:C_b(H^1(\mathbb{S}^2))\rightarrow C_b(H^1(\mathbb{S}^2))$ for all $t>0$.
	  \end{theorem}
	  \begin{proof}
	  	Assume indeed that $(x^n)_n$ converges to $x$ in $H^1(\mathbb{S}^2)$ (note that for each converging sequence $(x^n)_n$, there exists a radius $R>0$ so that $(x^n)_n\subset B_R(0)$). For every $\phi\in C_b(H^1(\mathbb{S}^2))$, it follows from the dominated convergence theorem and the continuity in \eqref{eq:cont_init_data_1} that
	  	\begin{align}\label{eq:continuity_phi}
	  	\lim_{n\rightarrow +\infty} (P_t\phi)(x^n)=\lim_{n\rightarrow +\infty} \mathbb{E}\left[\phi(u_t^{x^n})\right]=\mathbb{E}\left[\lim_{n\rightarrow +\infty}\phi(u_t^{x^n})\right]=\mathbb{E}\left[\phi(u_t^{x})\right]=(P_t\phi)(x)\, ,
	  	\end{align}
	  	which means that $P_t\phi:H^1(\mathbb{S}^2)\rightarrow H^1(\mathbb{S}^2)$ is continuous for every fixed $t>0$ and $\phi\in C_b(H^1(\mathbb{S}^2))$.  Thus $P_t\phi\in C_b(H^1(\mathbb{S}^2))$ (where the boundedness follows from the boundedness of $\phi$).
	  \end{proof}
	 For the reader's convenience, we briefly proof of \eqref{eq:cont_init_data_1} in the following Lemma \ref{lemma:feller_property}: we mainly highlight the parts of the proof where the anisotropic term and the boundary conditions play a role. For a more detailed proof see \cite[Theorem 5.5]{LLG1D}
	 \begin{lemma}\label{lemma:feller_property}
	 	The solution map $u=\Phi(u^0)$ to \eqref{LLG} is locally Lipschitz continuous as a function of $u^0\in H^1(\mathbb{S}^2)$ for every fixed $T>0$. Namely, fix $R>0$ and let $u^0,v^0\in B_R(0):=\{x\in H^1(\mathbb{S}^2):\|x\|_{H^1(\mathbb{S}^2)}< R\}$. Denote the respective solutions to \eqref{LLG} by $u=\Phi(u^0)$ and $v=\Phi(v^0)$, then
	 	\begin{align*}
	 		\sup_{t\in[0,T]}\|u_t-v_t\|^2_{H^1}\lesssim_{g,|D|}\exp\left(R^kT^k+R^k\omega_{W}^{1/p} (0,T)\right) \|u^0-v^0\|^2_{H^1}\, ,
	 	\end{align*}
	 	for all $u^0,v^0\in B_R(0)$ and for some power $k\geq 4$.
	 \end{lemma}
 \begin{proof}
 	In this proof, the parameters $\lambda_1,\lambda_2$ do not play any role, therefore we simplify the computations by setting $\lambda_1=\lambda_2=1$.
 	As a consequence of the uniqueness proof (which we modify by employing Proposition \ref{pro:product}), we know that $u$ depends continuously on the initial data in $L^2$. More specifically, given two initial conditions $u^0,v^0\in H^1(\mathbb{S}^2)$, it follows from \cite[Theorem 4.1, equation (4.8)]{LLG1D} that
 	\begin{align}\label{eq:uniqueness_estimate}
 		\sup_{t\in[0,T]}\|u_t-v_t\|^2_{L^2}\lesssim_{g,|D|} \exp(T\omega_{\mathbf{W}}^{1/p}(0,T)+T\|u^0\|_{H^1(\mathbb{S}^2)}^4+\|v^0\|_{H^1(\mathbb{S}^2)}^4) \|u^0-v^0\|_{L^2}^2\, ,
 	\end{align}
 	which is actually local Lipschitz continuity of the solution with respect to the initial datum in $L^2$ (remark that the initial datum needs to belong to $H^1(\mathbb{S}^2)$, in order to achieve this estimate). We notice that the Lipschitz constant depends on the time exponentially: this implies in particular that we can not use this kind of estimate to study the long time behaviour of the equation. 
 	
 	Since the initial datum lies in $H^1(\mathbb{S}^2)$, we need to prove also continuity with respect to the initial datum of the gradient of the solution. Set $z:=u-v$. In  order to achieve the conclusion, we need to study the system of equations
 	\begin{align}
 	&\delta (z^{\otimes 2})_{s,t}
 	= 2\mathcal{D}_{s,t}(0,0)
 	+ (I+ \mathbb{I}+ \tilde{\mathbb{I}})(0,0)_{s,t}
 	+(z^{\otimes 2})^\natural_{s,t}\label{eq:00}
 	\\
 	&\delta (z\otimes \partial_xz)_{s,t}
 	=\mathcal{D}_{s,t}(1,0)+\mathcal{D}_{s,t}(0,1)
 	+ (I+ \mathbb{I}+ \tilde{\mathbb{I}})(0,1)_{s,t}
 	+(z\otimes \partial _x z)^\natural_{s,t}\label{eq:01}
 	\\
 	&\delta (\partial_xz\otimes z)_{s,t}
 	=\mathcal{D}_{s,t}(1,0)+\mathcal{D}_{s,t}(0,1)
 	+ (I+ \mathbb{I}+ \tilde{\mathbb{I}})(1,0)_{s,t}
 	+(\partial_xz\otimes z)^\natural_{s,t}\label{eq:10}
 	\\
 	&\delta (\partial_xz\otimes \partial_xz )_{s,t}
 	=2\mathcal{D}_{s,t}(1,1)
 	+ (I+ \mathbb{I}+ \tilde{\mathbb{I}})(1,1)_{s,t}
 	+(\partial_xz\otimes \partial_xz)^\natural_{s,t}\, ,\label{eq:11}
 	\end{align}
 	where we used the product formula in Proposition \ref{pro:product} and the following notations
 	\begin{align*}
 	&\mathcal{D}(l,m):=\int_{s}^{t}\partial^l_x [(u\times \partial_x^2 u-\partial_x^2 u+u|\partial_x u|^2)-(v\times \partial_x^2 v-\partial_x^2 v+v|\partial_x v|^2)]\otimes\partial^m_x z\dd r\\
 	&\quad\quad\quad\quad\quad+\int_{s}^{t}\partial^l_x [u\times g'(u)-v\times g'(v)]\otimes\partial^m_x z\dd r\\
 	&\quad\quad\quad\quad\quad-\int_{s}^{t}\partial^l_x [u\times (u\times g'(u))-v\times(v\times g'(v))]\otimes\partial^m_x z\dd r\, , \\
 	&I_{s,t}(l,m):=\partial^l_x(W_{s,t}z_s)\otimes \partial^m_x z_s +\partial^l_x z_s\otimes\partial^m_x (W_{s,t}z_s)\, ,\\
 	&\mathbb{I}_{s,t}(l,m):=\partial^l_x(\mathbb{W}_{s,t}z_s)\otimes \partial^m_x z_s +\partial^l_x z_s\otimes\partial^m_x (\mathbb{W}_{s,t}z_s)\, ,\\
 	&\tilde{\mathbb{I}}_{s,t}(l,m):=\partial^l_x (W_{s,t}z_s)\otimes\partial^m_x (W_{s,t}z_s)\, ,
 	\end{align*}
 	for $m,l \in \{ 0 , 1 \}$. 
 	We notice that if we test the system above by $\mathbf{1}=(1_{l=m})$, it is possible to simplify the equations. 
 	
 	We need to consider the above system of equation since, while estimating the remainder term $(\partial_x z \otimes \partial_x z)^\natural$, the remainders $(\partial_x z \otimes z)^\natural, (z \otimes \partial_x z)^\natural, (z \otimes  z)^\natural$ appear. As a difference from the classical It\^o-Stratonovich calculus, again in the estimation of the remainder, the mixed terms $(\partial_x z^i  \partial_x z^j)^\natural$ appear, for $i\neq j$: this more complicated setting is the price to pay to get the pathwise convergence. 
 	
 	We recall the structure of the noise $(W,\mathbb{W})$. In our context the first iterated integral $W$ is an anti-symmetric matrix, namely $W^T=-W$. As a consequence $W_{s,t}a\cdot a=0$ for all $a\in \mathbb{R}^3$. The second iterated integral can be decomposed as
 	\begin{align*}
 		\mathbb{W}_{s,t}=\frac{1}{2}W_{s,t} W_{s,t}+\mathbb{L}_{s,t}\, ,
 	\end{align*}
 	where $\mathbb{L}$ is an anti-symmetric matrix.
 	Thus for all $a\in \mathbb{R}^3$ it follows that
 	\begin{align*}
 		&W_{s,t}a\cdot a+a\cdot W_{s,t}a=0\, ,\\
 		&\mathbb{W}_{s,t}a\cdot a+a\cdot \mathbb{W}_{s,t}a+ W_{s,t}a\cdot W_{s,t}a=-W_{s,t}a\cdot W_{s,t}a+W_{s,t}a\cdot W_{s,t}a=0\, .
 	\end{align*}
 	These symmetries lead to the conclusion that, if we test the system by $\mathbf{1}$, we can reduce the noises of the equations as follows
 	\begin{align*}
 		\langle(I+ \mathbb{I}+ \tilde{\mathbb{I}})_{s,t}(0,0),\mathbf{1}\rangle&=0\, , \quad\langle(I+ \mathbb{I}+ \tilde{\mathbb{I}})_{s,t}(1,0),\mathbf{1}\rangle=0\, , \quad \langle(I+ \mathbb{I}+ \tilde{\mathbb{I}})_{s,t}(0,1),\mathbf{1}\rangle=0 \, ,\\
 		\langle I_{s,t}(1,1),\mathbf{1}\rangle &=\langle\partial_xW_{s,t}z_s\otimes \partial_x z_s +\partial_x z_s\otimes \partial_xW_{s,t}z_s,\mathbf{1}\rangle\, ,\\
 		\langle\mathbb{I}_{s,t}(1,1)+\tilde{\mathbb{I}}_{s,t}(1,1),\mathbf{1}\rangle&=\langle\partial_x\mathbb{W}_{s,t}z_s\otimes \partial_x z_s +\partial_x z_s\otimes\partial_x \mathbb{W}_{s,t}z_s
 		+\partial_x W_{s,t}z_s\otimes\partial_x W_{s,t}z_s,\mathbf{1}\rangle\\
 		&\quad+ \langle W_{s,t}\partial_x z_s\otimes\partial_x W_{s,t}z_s+\partial_x W_{s,t}z_s\otimes W_{s,t}\partial_xz_s,\mathbf{1}\rangle\, .
 	\end{align*}
 	In particular, one has that in \eqref{eq:00} the equation is deterministic and also the remainder term cancels (one already observes this fact in the uniqueness proof). In \eqref{eq:01} and \eqref{eq:10}, the noise is also vanishing as well as the drifts (from integrations by parts): thus also the mixed equations for $l=m$ are deterministic (which is consistent with the fact that the solution lies on the sphere and thus its derivative lies in the tangent plane). Finally, the noises in \eqref{eq:11} simplify and the equation tested by $\mathbf{1}$ in $L^2$, has the form
 	\begin{align}\label{eq:energy_feller}
 		\delta \langle \partial_x  z\otimes \partial_x  z,\mathbf{1}\rangle _{s,t}=\langle (\mathcal
 		D +I+\mathbb{I}+\tilde{\mathbb{I}})_{s,t}(1,1),\mathbf{1}\rangle +\langle (\partial_x  z)_{s,t}^{\natural,2},\mathbf{1}\rangle\,  .
 	\end{align}
 	By estimating equation \eqref{eq:energy_feller}, we achieve the conclusion that the gradient of the equation is continuous with respect to the initial data in $H^1(\mathbb{S}^2)$: the main parts to bound are the drift $\langle \mathcal
 	D,\mathbf{1}\rangle $, the noise terms $\langle ( I+\mathbb{I}+\tilde{\mathbb{I}})_{s,t}(1,1),\mathbf{1}\rangle $ and the remainder $\langle (\partial_x  z)_{s,t}^{\natural,2},\mathbf{1}\rangle$.
 	We briefly recall the estimate of the drift: we remember from the proof of existence of the solution that $u\in L^\infty(H^1)\cap L^2(H^2)\cap C(L^2)$ pathwise. We exploit this boundedness, which holds for every fixed time interval $[0,T]$, for $T>0$ (note that the bounds on the solution depend exponentially on time, on the initial condition $u^0\in H^1(\mathbb{S}^2)$ and on the dimension of the domain). We write the drift tested by $\textbf{1}$, 
 	\begin{align*}
 	\langle	\mathcal{D}_{s,t}(1,1),\mathbf{1}\rangle &= \int_{s}^{t}\int_{D} \partial_x[\partial_x^2 z_r+z_r|\partial_x u_r|^2+v_r(|\partial_x u_r|^2-|\partial_x v_r|^2)+z_r\times \partial^2_x u_r+v_r\times \partial^2_x z_r]\cdot \partial_x z_r \dd x \dd r\\
 	&\quad +  \int_{s}^{t}\int_{D}\partial_x[z_r\times g'(u_r)+v_r\times(g'(u_r-v_r))]\cdot \partial_x z_r\dd x\dd r\\
 	&\quad +  \int_{s}^{t}\int_{D}\partial_x[u_r\times (u_r\times g'(u_r))-v_r\times(v_r\times g'(v_r))]\cdot \partial_x z_r\dd x\dd r\, ,
 	\end{align*}
 	which from integration by parts and by using the null Neumann boundary conditions, lead to
 	\begin{align*}
 	\langle	\mathcal{D}_{s,t}(1,1),\mathbf{1}\rangle&=-\int_{s}^{t}\|\partial_x^2 z_r\|_{L^2}^2 \dd r +  \int_{s}^{t}\int_{D}\partial_x[z_r\times g'(u_r)+v_r\times g'(u_r-v_r)]\cdot \partial_x z_r\dd x\dd r\\
 	&\quad+\int_{s}^{t}\int_{D}\partial_x[u_r\times (u_r\times g'(u_r))-v_r\times(v_r\times g'(v_r))]\cdot \partial_x z_r\dd x\dd r\\
 	&\quad-\int_{s}^{t}\int_{D}[z_r|\partial_x u_r|^2\cdot \partial_x^2 z_r+v_r(|\partial_x u_r|^2-|\partial_x v_r|^2)\cdot \partial_x^2 z_r+z_r\times \partial^2_x u_r\cdot \partial^2_x z_r] \dd x \dd r\, .
 	\end{align*}
 	The estimate of the different terms of the drift follows from the one dimensional Agmon's inequality $\|z\|_{L^\infty}\lesssim_{|D|} \|z\|_{L^2}^{1/2}\|z\|_{H^1}^{1/2}$, 
 	\begin{align*}
 		&\int_{s}^{t}\int_{D}[z|\partial_x u|^2\cdot \partial_x^2 z+v(|\partial_x u|^2-|\partial_x v|^2)\cdot \partial_x^2 z+z\times \partial^2_x u\cdot \partial^2_x z] \dd x \dd r\\
 		&\quad \leq \frac{3}{4}\int_{s}^{t}\|\partial_x^2 z\|^2_{L^2} \dd r+\|z\|_{L^\infty([s,t],L^2)}\|z\|_{L^\infty([s,t],H^1)}\int_{s}^{t}\|\partial_x u\|_{L^4}^4 \dd r\\
 		&\quad \quad+\|\partial_x z\|^2_{L^\infty([s,t],L^2)}\int_{s}^{t}[\|\partial_x u\|_{L^2}\|\partial_x u\|_{H^1}+\|\partial_x v\|_{L^2}\|\partial_x v\|_{H^1}] \dd r\\
 		&\quad\quad +\|z\|_{L^\infty([s,t];L^2)}\|z\|_{L^\infty([s,t];H^1)}\int_{s}^{t}\|\partial_x^2 u\|^2_{L^2} \dd r\, .
 	\end{align*}
 	Hence we conclude from Young's inequality and the uniqueness inequality \eqref{eq:uniqueness_estimate} that
 	\begin{align*}
 		&\int_{s}^{t}\int_{D}[z|\partial_x u|^2\cdot \partial_x^2 z+v(|\partial_x u|^2-|\partial_x v|^2)\cdot \partial_x^2 z+z\times \partial^2_x u\cdot \partial^2_x z] \dd x \dd r\\
 		&\quad \leq \frac{3}{4}\int_{s}^{t}\|\partial_x^2 z\|^2_{L^2} \dd r+C_T\left[\|u^0\|^4_{H^1}+\|v^0\|^4_{H^1}\right](\|z\|^2_{L^\infty(s,t;H^1)}+\|z^0\|^2_{L^2})\, ,
 	\end{align*}
 	where $C_T>0$ is a constant depending exponentially on time, on the dimension of the domain and on some powers of the $H^1(\mathbb{S}^2)$ norm of the initial conditions $u^0,v^0\in H^1(\mathbb{S}^2)$ such that
 	\begin{align*}
 		\left[(1+T)(\|u\|^2_{L^\infty(H^1)\cap L^2(H^2)}+\|v\|^2_{L^\infty(H^1)\cap L^2(H^2)}  )\right]<C_T\left[\|u^0\|^4_{H^1}+\|v^0\|^4_{H^1}\right]\, .
 	\end{align*}
 	Here we used that $\|\partial_x u\|_{L^4}^4\lesssim \|\partial_x u\|_{H^1}\|\partial_x u\|^3_{L^2}$. With analogous computations, the first anisotropic part of the drift can be bounded by
 	\begin{align*}
 		\int_{s}^{t}\int_{D}\partial_x[u_r\times (u_r\times g'(u_r))-v_r\times(v_r\times g'(v_r))]\cdot \partial_x z_r\dd x\dd r\lesssim \|g'\|_{L^\infty}\int_{s}^{t}\|\partial_x z_r\|_{L^2}^2\dd r\, ,
 	\end{align*}
 	as well as the other anisotropic part. We now turn to the estimation of the noise $\langle ( I+\mathbb{I}+\tilde{\mathbb{I}})_{s,t}(1,1),\mathbf{1}\rangle $, which from Young's inequality and \eqref{eq:uniqueness_estimate} leads to
 	\begin{align*}
 		\langle ( I+\mathbb{I}+\tilde{\mathbb{I}})_{s,t}(1,1),\mathbf{1}\rangle &\leq \frac{\omega_{W}^{1/p}(s,t)}{2}\left[\|z\|_{L^\infty([s,t],L^2)}^2+\|z\|^2_{L^\infty([s,t],H^1)}\right]\\
 		&\lesssim \frac{\omega_{W}^{1/p}(s,t)}{2}\left[\|z^0\|^2_{L^2}+\|z\|^2_{L^\infty([s,t],H^1)}\right]\, .
 	\end{align*}
 	We now pass to the estimate of the remainder $\langle (\partial_x  z)_{s,t}^{\natural,2},\mathbf{1}\rangle$: this term depends on the mixed equations \eqref{eq:01}, \eqref{eq:10} and on the first level equation \eqref{eq:00}. In particular, for technical reasons, we estimate $\|(\partial_x  z)_{s,t}^{\natural,2}\|_{H^{-1}}$: this allows to conclude that $\langle (\partial_x  z)_{s,t}^{\natural,2},\mathbf{1}\rangle \leq \|(\partial_x  z)_{s,t}^{\natural,2}\|_{H^{-1}}\|\mathbf{1}\|_{H^1}$. We employ the sewing Lemma \ref{lemma_sewing}, after applying the operator $\delta(f)_{s,u,t}:=f_{s,t}-f_{s,u}-f_{u,t}$ to each component $(\partial_x  z)^{\natural, 2}$, we obtain
 	\begin{align*}
 	\|(\partial_x  z)^{\natural,2}_{s,t}\|_{H^{-1}}\lesssim & \;  \omega_{\mathbf{W}}^{1/p}(s,t)\big[\| z\|_{L^\infty (s,t;H^1)}^2+[\omega_{\mathcal{D}(1,0);H^{-1}}+\omega_{\mathcal{D}(0,1);H^{-1}}+\omega_{\mathcal{D} (0,0);L^2}](s,t)\big]\, ,
 	\end{align*}
	where the $H^{-1}$ norm in $\mathcal{D}(1,0)$ is bounded by $\mathcal{D}(0,1)$. By noticing that from the uniqueness proof we can bound $\omega_{\mathcal{D} (0,0);L^2}\lesssim_T \|z^0\|_{L^2}^2$
	and from similar computations as for $\mathcal{D}(1,1)$, we can conclude that
	\begin{align*}
		\omega_{\mathcal{D} (0,1);L^2}(s,t)\lesssim \frac{\epsilon}{2}\int_{s}^{t}\|\partial_x^2 z_r\|_{L^2}^2 \dd r+C_T\left[\|u^0\|^4_{H^1}+\|v^0\|^4_{H^1}\right]\omega_{\mathbf{W}}^{1/p} (s,t)[\|z\|^2_{L^\infty(s,t;H^1)}+\|z^0\|_{L^2}^2]\, .
	\end{align*}
	In conclusion, by choosing $\epsilon>0$ suitably small, we can rewrite the equation as
	\begin{equation}\label{eq:why_RP_matter}
	\begin{aligned}
	\delta \left[\|\partial_x  z\|^{2}_{L^2} \right]_{s,t} + \frac{1}{8}\int_{s}^{t} &\|\partial_x 
	^{2} z_r\|^2_{L^2} \dd r\lesssim \|\partial_x z^0\|^2_{L^2}\\
	&\quad+C_T\left[1+\omega_{\mathbf{W}}^{1/p}\right]\left[\|u^0\|^4_{H^1}+\|v^0\|^4_{H^1}\right]\left[\|z\|^2_{L^\infty(s,t;H^1)}+\|z^0\|^2_{L^2}\right]\, ,
	\end{aligned}
	\end{equation}
	which from the rough Gronwall's Lemma \ref{lem:gronwall} leads to the conclusion.
 \end{proof}
	\begin{remark} \label{remark:stratonovich_not_possible_feller}
		\textbf{On Feller's property with classical Stratonovich calculus.}
		For every fixed $T>0$ it is possible, by means of the rough path theory, to prove the local Lipschitz continuity of the solution map with respect to the initial condition in $H^1(\mathbb{S}^2)$  $\mathbb{P}$-a.s. We show that it does not seem possible to conclude the same by means of the classical It\^o-Stratonovich calculus.
		We keep the notations of the proof of Lemma \ref{lemma:feller_property}. Assume that $u,v$ are martingale solutions to \eqref{LLG} with initial conditions $u^0, v^0\in \mathcal{L}^4(\Omega;H^1(\mathbb{S}^2))$. We apply the It\^o's formula to the equation for the difference $\partial_x z=\partial_x(u-v)$
		\begin{align*}
			 \|\partial_x z_t\|_{L^2}^2-\|\partial_x z^0\|_{L^2}^2&=\langle\mathcal{D}_{0,t}(1,1)[u]-\mathcal{D}_{0,t}(1,1)[v],\mathbf{1}\rangle+\int_{D}\int_{0}^{t} \partial_x hz_r\times  \dd W_r \dd x\\
			 &\quad\quad\quad\quad+2 \int_{0}^{t}\int_{D} \partial_x h^2 [z\cdot z- z\cdot \partial_x z]\dd x\dd r
			+2\int_{D}\int_{0}^{t}\partial_x h\partial_x z\cdot  z\times  \dd W_r\dd x\, .
		\end{align*}
		If we first look at the drift estimate, we can conclude with the same estimates as in Lemma \ref{lemma:feller_property}, more specifically
		\begin{align*}
			\langle\mathcal{D}_{0,t}(1,1)[u]-\mathcal{D}_{0,t}(1,1)[v],\mathbf{1}\rangle\leq -\frac{1}{4}\int_{0}^{t}\|\partial_x^2 z\|^2_{L^2}\dd r+C_T\left[\|u^0\|^4_{H^1}+\|v^0\|^4_{H^1}\right]\left(\|z\|^2_{L^\infty(H^1)}+\|z^0\|^2_{L^2}\right).
		\end{align*}
		Since we look for the $L^\infty (H^1(\mathbb{S}^2))$-norm of $z$, we take the supremum in time into the equation and we estimate the noises by the Burkholder-Davis-Gundy inequality. To do so, we take the expectation of the energy and, from the estimate of the drifts, we conclude that we can not apply Gronwall's Lemma to achieve the required bound. This is due to the elements $\left[\|u^0\|^4_{H^1}+\|v^0\|^4_{H^1}\right]$, which do not allow to pass to get the correct powers for the Grownall's Lemma.
	\end{remark}	 
	 
\subsection{Tightness of $(\mu_T)_{T>0}$ in $H^1(\mathbb{S}^2)$ via Stratonovich calculus. }\label{sec:tightness}
In Lemma \ref{lemma:u_dot_nablau_0} we observe an orthogonality property for the solution $u$ to \eqref{LLG}, which leads to an equality useful in the following.
\begin{lemma}\label{lemma:u_dot_nablau_0}Let $u\in L^\infty(H^1)$ such that $|u_t(x)|_{\mathbb{R}^3}=1\,$ for a.e. $(t,x)\in[0,T]\times D\,$, then
	\begin{align*}
	u_t(x)\cdot \partial_x u_t(x)=0\quad a.e.\,\, (t,x)\in[0,T]\times D\, .
	\end{align*}
	In particular, it follows that $|u_t(x)\times \partial_x u_t(x)|_{\mathbb{R}^3}=|\partial_xu_t(x)|_{\mathbb{R}^3}$ for a.e. $(t,x)\in[0,T]\times D\,$.
\end{lemma}
\begin{proof}
	Recall that $|u_t(x)|=1$ for a.e. $(t,x)\in[0,T]\times D$. Then from the product rule for Sobolev functions (since $u_t\in H^1$) it follows that for a.e. $t\in[0,T]$
	\begin{align*}
	\partial_x |u_t|^2=2\partial_x u_t\cdot u_t\, .
	\end{align*}
	We apply the product rule in the first equality and we observe that
	\begin{align*}
	2\int_{D} \phi u\cdot\partial_x u  d x=\int_{D} \phi \partial_x |u|^2  d x=-\int_{D} \partial_x\phi |u|^2 dx=0\, ,
	\end{align*}
	for every $\phi\in C^1_0(D)$ (the equality on the right hand side is $0$ since $|u|=1$ for a.e. $(t,x)\in [0,T]\times D$ and $\phi$ is compactly supported on $D$): thus from the fundamental lemma of calculus of variations it follows that $u_t(x)\cdot\partial_x u_t(x)=0$ for a.e. $(t,x)\in [0,T]\times D$. 
	As a consequence, $u$ is orthogonal to $\partial_x u$ and $|u_t(x)\times \partial_x u_t(x)|=|u_t(x)||\partial_x u_t(x)||\sin(\pi/2)|=|\partial_x  u_t(x)|$ for a.e. $(t,x)\in [0,T]\times D$.
\end{proof}
We derive first a linear bound in time for the gradient norm of the solution, in absence of anisotropic energy.
\begin{lemma} \label{lemma:lemma_sup_grad} Let $u$ be the unique solution to \eqref{LLG} in the sense of Definition \ref{def:solution} with $g\equiv 0$. For every $t>0$, the bound holds
	\begin{equation}\label{eq:unif_bound_time}
	\begin{aligned}
	\sup_{r\in[0,t]}\mathbb{E}\left[\|\partial_x u_r\|^2_{L^2}\right]+2\lambda_2\int_{0}^{t}\mathbb{E}\left[\|u_r\times\partial^2_x u_r\|^2_{L^2}\right]\dd r\leq\mathbb{E}\left[\|\partial_x u^0\|_{L^2}^2\right]+t\|\partial_x h\|^2_{L^2}\, .
	\end{aligned}
	\end{equation}
\end{lemma}
\begin{proof}
	The rough integral coincides with the Stratonovich stochastic integral apart from a set of null measure, therefore we can switch to the classical It\^o-Stratonovich calculus. We write the equation for the derivative in Stratonovich form
	\begin{align*}
	\delta \partial_x u_{s,t}=\int_{s}^{t}\partial_x[\lambda_1 u_r\times \partial_x^2 u_r-\lambda_2 u_r\times (u_r\times \partial_x^2 u_r)] \dd r+\int_{s}^{t}h\partial_x u_r\times \circ \dd W_r+\int_{s}^{t} \partial_x hu_r\times \circ \dd W_r\, .
	\end{align*}
	First we convert the equation from Stratonovich integration to It\^o integration. We would like to determine $c(x)\equiv(c_1(x),c_2(x),c_3(x))$ defined for all $i=1,2,3$ by
	\begin{align*}
	c_i(x)=\sum_{k=1}^{3}\sum_{j=1}^{3}\frac{\partial \gamma_{i,j}}{\partial x_j}(x)\gamma_{j,k}\, ,
	\end{align*}
	where $x\equiv(x_1,x_2,x_3)$ and the map  $\gamma(x)\equiv (\gamma_{i,j}(x))_{i,j=1,2,3}$ is given by
\[
\gamma(x)=x\times \cdot=
\left[ {\begin{array}{ccc}
	0 & -x_3&x_2\\
	x_3 & 0&-x_1\\
	-x_2& x_1 &0\\
	\end{array} } \right]
\]
where $x\equiv(x_1,x_2,x_3)\in \mathbb{R}^3$. Thus we can rephrase the integrals as
	\begin{align*}
	\int_{s}^{t}\gamma(u_r)\circ\dd W_r=\frac{1}{2}\int_{s}^{t} c(u_r) dr+\int_{s}^{t}\gamma(u_r)dW_r\, .
	\end{align*}
	By using the above formula, we conclude that $c(x)=[\gamma_{2,3}-\gamma_{3,2},\gamma_{3,1}-\gamma_{1,3},\gamma_{1,2}-\gamma_{2,1}] (x)$, which leads to $c(X)=-2X$. Hence
	\begin{align*}
		\int_{s}^{t}h\partial_x u_r\times \circ \dd W_r+\int_{s}^{t} \partial_x h u_r\times \circ \dd W_r&=\int_{s}^{t}h\partial_x u_r\times  \dd W_r-\int_{s}^{t}h^2\partial_x u_r \dd r\\
		&\quad +\int_{s}^{t} \partial_x h u_r\times \dd  W_r-\int_{s}^{t} \partial_x h^2 u_r \dd r\, .
	\end{align*}
	From It\^o's formula applied to $f(X)\equiv f(X_1,X_2,X_3)=X\cdot X$, where we use that $\nabla_Xf(X)=2X$ and $\nabla^2_Xf(X)=2I\in \mathbb{R}^3\otimes\mathbb{R}^3$. This leads to
	\begin{align*}
	\|\partial_x u_t\|^2_{L^2}+2\lambda_2\int_{0}^{t}\|u\times \partial_x^2 u\|_{L^2}^2\dd r&=\|\partial_x u^0\|^2_{L^2}-2\int_{0}^{t}\int_{D} h^2\partial_x u\cdot \partial_x u\dd x\dd r-2\int_{0}^{t}\int_{D}\partial_x h^2u\cdot \partial_x u\dd x\dd r\\
	&+2 \int_{0}^{t}\int_{D} h^2\partial_x u\cdot \partial_x u\dd x\dd r+2 \int_{0}^{t}\int_{D} \partial_x h^2 u\cdot u\dd x\dd r\\
	& +2\int_{D}\int_{0}^{t}h\partial_x u\cdot\partial_x u\times  \dd W_r \dd x +2\int_{D}\int_{0}^{t}\partial_x h\partial_x u\cdot  u\times  \dd W_r\dd x\, .
	\end{align*}
	By taking expectation into the above equation, the energy inequality takes the form
	\begin{align}
	\mathbb{E}\left[\|\partial_x u_t\|_{L^2}^2\right]&+2\lambda_2\int_{0}^{t}\mathbb{E}\left[\|u\times \partial_x^2 u\|^2_{L^2}\right] \dd r=\mathbb{E}\left[\|\partial_x u_0\|_{L^2}^2\right]+\mathbb{E}\left[2 \int_{0}^{t}\int_{D} \partial_x h^2 u\cdot u \dd x\dd r\right]\label{eq:energy_0}\\
	&\quad\quad\quad+\mathbb{E}\left[2\int_{D}\int_{0}^{t}h\partial_x u\cdot\partial_x u\times  \dd W_r \dd x +2\int_{D}\int_{0}^{t}\partial_x h\partial_x u\cdot  u\times  \dd W_r\dd x\right]\label{eq:energy_1}\, .
	\end{align}
	The integral in \eqref{eq:energy_0} is deterministic, since $|u_t(x)|^2_{\mathbb{R}^3}=1$ for a.e. $(x,t)\in D\times [0,T]$ and $\mathbb{P}$-a.s. and therefore we rewrite it as 
	\begin{align}
	\mathbb{E}\left[2 \int_{0}^{t}\int_{D} \partial_x h^2 u_r\cdot u_r \dd x\dd r\right]=2t\|\partial_x h\|_{L^2}^2\, .
	\end{align}
	The stochastic integrals in \eqref{eq:energy_1} are an It\^o integrals and thus it have null expectation, which concludes the proof.
	\end{proof}
	We introduce now Lemma \ref{lemma:poincarre}, which we need in Lemma \ref{lemma:energy_anisotropic}.
	\begin{lemma}\label{lemma:poincarre}Let $C_p>0$ be the Poincaré's constant associated to $D$ and $u$ be the unique solution to \eqref{LLG},  then the inequality holds
	\begin{align*}
	C_p^{-1}\|u\times\partial_x u\|_{L^2}\leq	\|u\times\partial^2_x u\|_{L^2}\, .
	\end{align*}
	\end{lemma}
	\begin{proof}
	As a consequence of Lemma \ref{lemma:u_dot_nablau_0}, it follows that $\|\partial_x u\|_{L^2}^2=\|\partial_x u\times u\|_{L^2}^2$ for a.e. $t>0$ and $\mathbb{P}-$a.s. We also observe that, from $a\times a=0$ for all $a\in\mathbb{R}^3$, it holds in $L^2$ that 
	\begin{align*}
	\partial_x (u\times \partial_x u)=\partial_x u\times \partial_x u+u\times \partial_x^2 u=u\times \partial_x^2 u\, .
	\end{align*}
	Since we are dealing with a one dimensional domain, from Morrey's inequality $H^1(\mathbb{S}^2)$ is continuously embedded into $C(D;\mathbb{R})$: as a consequence both $u_t$ and $\partial_x u_t$ are continuous in the space variable for a.e. $t\geq 0$ and $\mathbb{P}$-a.s. Thus we can infer that we can extend the derivative $\partial_x u$ continuously to $0$ on the boundary. Because of the continuity of $\partial_x u$ on the boundary and of  $u$ on the boundary, also $u\times \partial_x u$ is null on the boundary: thus $u\times \partial_x u\in H^1_0$. We are therefore in the conditions to apply Poincaré's inequality: there exists $C_p>0$ such that 
	\begin{align*}
	 	C_p^{-1}\|u\times\partial_x u\|_{L^2}\leq\|\partial_x (u\times\partial_x u)\|_{L^2}=	\|u\times\partial^2_x u\|_{L^2}\, .
	\end{align*}
	which leads to the conclusion.
\end{proof}
We add the anisotropic energy and observe how inequality \eqref{eq:unif_bound_time} changes.
\begin{lemma}\label{lemma:energy_anisotropic} For every $t>0$ and for a positive constant $C(\lambda_1,\lambda_2)>0$, it holds 
	\begin{equation}
	\begin{aligned}\label{eq:unif_bound_time_1}
	\sup_{r\in[0,t]}\mathbb{E}\left[\|\partial_x u_r\|^2_{L^2}\right]+&\frac{3\lambda_2}{2}\int_{0}^{t}\mathbb{E}\left[\|u_r\times\partial^2_x u_r\|^2_{L^2}\right]\dd r\\
	&\quad\quad\quad\leq\mathbb{E}\left[\|\partial_x u^0\|_{L^2}^2\right]+t\left[\|\partial_x h\|^2_{L^2}+[\sup_{i,j}|A_{i,j}|^2+|b|^2]C(\lambda_1,\lambda_2)\right]\, .
	\end{aligned}
	\end{equation}
\end{lemma}
\begin{proof}
Recall the shape of the anisotropic energy: for $A\in \mathcal{L}(\mathbb{R}^3)$ and $b\in\mathbb{R}^3$ as $g'(x)=Ax+b$, for all $x\in \mathbb{R}^3$. We turn to the drift elements appearing in \eqref{LLG}, which from the orthogonality in \eqref{eq:rel_a_laplace} leads to
\begin{align*}
	2\lambda_1\int_{0}^{T}\int_{D}\partial_x(u_r\times g'(u_r))\cdot \partial_x u_r \dd x \dd r&= -2\lambda_1\int_{0}^{T}\int_{D}(u_r\times g'(u_r))\cdot \partial_x^2u_r \dd x \dd r\\
	&= -2\lambda_1\int_{0}^{T}\int_{D}(u_r\times g'(u_r))\cdot (u_r\times\partial_x^2u_r )\dd x \dd r\, .
\end{align*}
In the case $\lambda_1=0$, the other term does not appear. From the weighted Young's inequality, we can bound this term for $\epsilon>0$
\begin{align*}
2\lambda_1\int_{0}^{T}\int_{D}\partial_x(u_r\times g'(u_r))\cdot \partial_x u_r \dd x \dd r\leq \frac{2\lambda_1^2\epsilon}{4}\int_{0}^{T}\|u_r\times\partial_x^2 u_r\|^2_{L^2}\dd r+\frac{2\cdot 4}{ 3\epsilon}\int_{0}^{T}\|u_r\times g'(u_r)\|^2_{L^2} \dd r\, ,
\end{align*}
where we use $\epsilon=\lambda_2/\lambda_1^2$. With analogous considerations and $\epsilon=1/\lambda_2$, we obtain the bound
\begin{align*}
2\lambda_2\int_{0}^{T}\int_{D}\partial_x(u\times (u\times g'(u)))\cdot \partial_x u_r \dd x \dd r\leq \frac{2\lambda_2}{4}\int_{0}^{T}\|u\times\partial_x^2 u\|^2_{L^2}\dd r+\frac{2\cdot 4\lambda_2}{ 3}\int_{0}^{T}\|u\times g'(u)\|^2_{L^2} \dd r\, .
\end{align*}
\end{proof}
\begin{remark} \label{ref:different_choice_noise}\textbf{On a different choice of the noise.}
Assume now that equation \eqref{LLG} is driven by the noise
\begin{align*}
\int_{0}^{t}u_r\times h\circ \dd W_r\, ,
\end{align*}
where $h\in H^1(D;\mathbb{R}^3)$ and $W$ is  a real valued Brownian motion. The It\^o formula leads, in this case, to a different outcome. Nevertheless the final estimate coincides, up to a non relevant constant, to \eqref{eq:unif_bound_time_anisotrpoic}. The only difference is that one needs to pass through the Poincaré's inequality to absorb to the left hand side
\begin{align*}
\int_{0}^{t}\int_{D} |h||\partial_x u_r||\partial_x h||u_r|\dd x\dd r\, .
\end{align*}
This fact holds true also in absence of anisotropic energy.
\end{remark}

\begin{lemma}\label{lemma:relazione_laplaciano}
	It holds for a.e. $t>0$ and $\mathbb{P}$-a.s. that
	\begin{align*}
	\|\partial^2_x u\|_{L^2}^2=\|\partial_x u\|_{L^4}^4+\|u\times\partial_x^2 u\|^2_{L^2}\, .
	\end{align*}
\end{lemma}
\begin{proof} We recall from Lemma 2.4 in \cite{brzezniak_LDP}. It holds in $L^2$ that
	\begin{align}\label{eq:rel_a_laplace}
	-u\times(u\times \partial^2_x u)=\partial_x^2 u+u|\partial_xu|^2\, .
	\end{align}
	Since $\partial^2_x u\in L^2$, we can test \eqref{eq:rel_a_laplace}
	\begin{align*}
		-\int_{\mathbb{T}^1}u\times(u\times \partial^2_x u)\cdot \partial^2_x u \dd x=\int_{\mathbb{T}^1} (\partial_x^2 u+u|\partial_xu|^2)\cdot \partial^2_x u \dd x\, ,
	\end{align*}
	By recalling that $a\times(a\times b)\cdot b=-|a\times b|^2$ for all $a,b\in\mathbb{R}^3$ and by using that $|u|^2=1$, we conclude that
	\begin{align*}
	\|\partial_x^2 u\|^2_{L^2}=\|u\times\partial_x^2 u\|^2_{L^2}+\|\partial_x u\|^4_{L^4}\, ,
	\end{align*}
	where we used the equality $|\partial_x u|^2=-\partial_x^2 u\cdot u$ for a.e. $t>0$, $x\in D$ and $\mathbb{P}$-a.s.
\end{proof}

\begin{lemma} \label{lemma:linear_growth_tightness}There exists a constant $C\equiv C(|D|,\lambda_1,\lambda_2,\bar{G},\|u^0\|_{H^1})>0$ independent on the time, such that for all $t>0$
\begin{align}\label{eq:estimate_laplaciano_time}
\int_{0}^{t}\mathbb{E}\left[\|\partial_x^2 u_r\|^{1/2}_{L^2}\right] \dd r \leq C\mathbb{E}\left[\|\partial_x u^0\|_{L^2}^2\right]+Ct\, .
\end{align}
\end{lemma}
\begin{proof}
From Lemma \ref{lemma:relazione_laplaciano}, it holds for a.e. $t>0$ fixed and $\mathbb{P}$-a.s. that
\begin{align*}
	\|\partial^2_x u_t\|_{L^2}^2=\|\partial_x u_t\|_{L^4}^4+\|u\times\partial_x^2 u_t\|^2_{L^2}\, .
\end{align*}
By employing the previous relation, by taking the power $1/4$ it follows that
\begin{align}\label{eq:stima_laplacie_1_2}
	\|\partial^2_x u\|^{1/2}_{L^2}=(\|\partial^2_x u\|_{L^2}^2)^{1/4}=(\|\partial_x u\|_{L^4}^4+\|u\times\partial_x^2 u\|^2_{L^2})^{1/4}\leq (\|\partial_x u\|_{L^4}+\|u\times\partial_x^2 u\|^{1/2}_{L^2})\, .
\end{align}
Thus we need to estimate the norms on the right hand side. Recall the one dimensional interpolation inequality
\begin{align*}
\|z\|_{L^4}\leq C(D)\|z\|_{L^2}^{3/4}\|z\|_{H^1}^{1/4}\, ,
\end{align*}
where $C(D)>0$ is a constant depending only on the dimension of the domain.
We integrate in time in \eqref{eq:stima_laplacie_1_2} and we estimate first the $L^4$-norm: from Hölder's inequality and the one dimensional interpolation inequality 
\begin{align*}
\int_{0}^{t}\|\partial_x u_r\|_{L^4} \dd r\leq C(D) \int_{0}^{t}\|\partial_x u_r\|_{L^2}^{3/4}\|\partial_x^2 u_r\|_{L^2}^{1/4} \dd r\, . 
\end{align*}
By taking expectation and integrating in time in \eqref{eq:stima_laplacie_1_2} and from the weighted Young's inequality with $\epsilon>0$
\begin{align*}
		\int_{0}^{t}\mathbb{E}[\|\partial_x^2 u_r\|^{1/2}_{L^2}] \dd r\leq  \frac{\epsilon C(D)}{2} \int_{0}^{t}\mathbb{E}[\|\partial_x u_r\|_{L^2}^{3/2}] \dd r+\frac{C(D)}{2\epsilon}\int_{0}^{t}\mathbb{E}[\|\partial_x^2 u_r\|^{1/2}_{L^2}]\dd r +\int_{0}^{t}\mathbb{E}[\|u_r\times\partial_x^2 u_r\|^{1/2}_{L^2}]\dd r \, .
\end{align*}
By choosing $\epsilon=C(D)$ , we can absorb the Laplacian to the left hand side and we obtain
\begin{align*}
\frac{1}{2}\int_{0}^{t}\mathbb{E}[\|\partial_x^2 u_r\|^{1/2}_{L^2}] \dd r\leq  \frac{C(D)^2}{2} \int_{0}^{t}\mathbb{E}[\|\partial_x u_r\|_{L^2}^{3/2}] \dd r +\int_{0}^{t}\mathbb{E}[\|u_r\times\partial_x^2 u_r\|^{1/2}_{L^2}]\dd r \, .
\end{align*}
 From Young's inequality and Lemma \ref{lemma:lemma_sup_grad} or Lemma \ref{lemma:energy_anisotropic} (which Lemma to use, depends on the presence or not of anisotropic energy: this affects the constant $C$ in the statement), it follows that
\begin{align*}
\int_{0}^{t}\mathbb{E}[\|\partial_x u_r\|_{L^2}^{3/2}] \dd r\lesssim \int_{0}^{t}\mathbb{E}[\|\partial_x u_r\|_{L^2}^{2}]\dd r+Ct\lesssim C\mathbb{E}[\|\partial_x u^0\|_{L^2}^2]+Ct\, ,
\end{align*}
\begin{align*}
\int_{0}^{t}\mathbb{E}[\|u_r\times\partial_x^2 u_r\|^{1/2}_{L^2}]\dd r\lesssim \int_{0}^{t}\mathbb{E}[\|u_r\times\partial_x^2 u_r\|^2_{L^2}]\dd r+Ct\lesssim C\mathbb{E}[\|\partial_x u^0\|_{L^2}^2]+Ct\, ,
\end{align*}
which concludes the proof.
\end{proof}

\begin{lemma}\label{lemma:tightness_mu_t}
The sequence $(\mu_T)_{T}$ is tight in $H^1(\mathbb{S}^2)$.
\end{lemma}
\begin{proof}
The space $H^2(\mathbb{S}^2)$ is compactly embedded in $H^1(\mathbb{S}^2)$, therefore the ball $B_R:=\{x\in H^2(\mathbb{S}^2):\|x\|_{H^2(\mathbb{S}^2)}\leq R\}$ is compact in $H^1(\mathbb{S}^2)$, for some $R>0$. We use this compact set to prove the tightness of $(\mu_T)_{T>0}$, where $\mu_T$ is defined in \eqref{eq:mu_T}. By evaluating each $\mu_T$ in $B_R^C:=H^2(\mathbb{S}^2) \setminus B_R$, 
\begin{align}\label{eq:tightnees}
\mu_T(B^C_R)=\mu_T(\|x\|_{H^2}>R)=\frac{1}{T}\int_{0}^{T}\mathbb{P}(\|u^x_t\|_{H^2(\mathbb{S}^2)}>R)\dd t=\frac{1}{T}\int_{0}^{T}\mathbb{P}(\|u^x_t\|^{1/2}_{H^2(\mathbb{S}^2)}>\sqrt{R})\dd t\, ,
\end{align}
where we used that $\phi(w)=\sqrt{w}$ is monotone increasing.
From Markov's inequality applied to \eqref{eq:tightnees} with the positive non-decreasing function $\phi(w)$ and from the estimate in Lemma \ref{lemma:linear_growth_tightness},
\begin{align*}
\mu_T(B^C_R)\leq \frac{1}{T\sqrt{R}}\int_{0}^{T}\mathbb{E}\left[\|u^x_t\|_{H^2}^{1/2}\right]\dd t \leq \frac{C(1+T)}{\sqrt{R}T}\leq \frac{2C}{\sqrt{R}}\, ,
\end{align*}
where we used that $T>1$. By taking the limit for $R\rightarrow +\infty$, we conclude that $\mu_T(B^C_R)$ converges to $0$ and thus that $(\mu_T)_{T>0}$ is tight in $H^1(\mathbb{S}^2)$.
\end{proof}

\subsection{Existence of an invariant measure.}\label{sec:inv_measure_existence}
In Theorem \ref{teo:existence_invariant_measure} we conclude that there exists an invariant measure for the semigroup $(P_t)_t$.
\begin{theorem}\label{teo:existence_invariant_measure}
There exists at least an invariant measure for the semigroup $(P_t)_t$ associated to \eqref{LLG} on $H^1(\mathbb{S}^2)$. 
\end{theorem}
\begin{proof}
As a consequence of Lemma \ref{lemma:feller_property}, the semigroup $(P_t)_t$ has the Feller property in $H^1(\mathbb{S}^2)$. From Lemma~\ref{lemma:tightness_mu_t}, $(\mu_T)_T$ is tight in $H^1(\mathbb{S}^2)$: from the Krylov-Bogoliubov Theorem \ref{teo:Krylov_B}, there exists at least an invariant measure $\mu$ on $H^1(\mathbb{S}^2)$ associated to the semigroup $(P_t)_t$. 
\end{proof}
\begin{remark}
We prove existence of a stationary solution as limit in the weak-star topology of a subsequence in \eqref{eq:mu_T}. This implies that the set of invariant measures constructed as weak-star limits of subsequences of the form \eqref{eq:mu_T} is not empty. A posteriori, we observe that every invariant measure can be built with the Krylov-Bogoliubov procedure, since \eqref{eq:mu_T} reduces to the constant sequence for an invariant measure. Thus the set of invariant measures for $(P_t)_t$ coincides with the set con measures constructed by means of \eqref{eq:mu_T}.
\end{remark}

\section{Ergodic measures and stationary solutions}\label{sec:ergodic_inv_measure}
\subsection{Existence of  stationary solutions}
We look at stationary pathwise solutions to the equation. Given an invariant measure $\mu\in \mathcal{I}$, there exists a random variable $w^0$ distributed like $\mu$ on a probability space $(\tilde{\Omega},\tilde{\mathcal{F}},\tilde{\mathbb{P}})$. From Skorohod's representation theorem, there exists a filtered probability space $(\Omega,\mathcal{F},(\mathcal{F}_t)_t,\mathbb{P})$ where the Brownian motion $W$, the initial condition $u^0$ and the initial condition $w^0$ are adapted. With abuse of notation, we will not distinguish between the different probability spaces. We discuss the regularity of stationary solutions.
\begin{theorem}\label{teo:pathwise_stationary_sol}
	Let $w^0$ be an initial condition distributed like an invariant measure $\mu$ of the semigroup $(P_t)_t$ on $H^1(\mathbb{S}^2)$. Then there exists a pathwise stationary solution $w$ such that $w(\omega)\in L^\infty(H^1)\cap L^2(H^2)\cap C([0,T];L^2)$ $\mathbb{P}-$a.s.	
\end{theorem}
Theorem \ref{teo:pathwise_stationary_sol} states that for each invariant measure $\mu$ the equation admits a pathwise stationary solution. We still do not have informations on the integrability with respect to the probability space. If we use the classical Stratonovich calculus, we need to require $u^0\in \mathcal{L}^4(\Omega;H^1(\mathbb{S}^2))$. We are not able to prove that there exists an initial condition $w^0 $ distributed like an invariant measure $\mu$ such that $w^0\in \mathcal{L}^4(\Omega;H^1(\mathbb{S}^2))$, which means that we do not know whether there exists a stationary solution that can be interpreted as a martingale solution. We can nevertheless prove that $w^0\in \mathcal{L}^2(\Omega;H^1(\mathbb{S}^2))$  and interpret stationary solution as pathwise solutions in the sense of Definition \ref{def:solution}.

\begin{theorem}\label{teo:regolarity_pathwise_stationary_sol}
	Assume $u^0\in \mathcal{L}^2(\Omega ;H^1(\mathbb{S}^2))$ and let $w^0$ be distributed like an invariant measure $\mu$ to $(P_t)_t$. Every stationary solution $w$ started in $w^0$ has bounded second moment in $H^1(\mathbb{S}^2)$. The constant $K>0$ is common to every stationary solution, i.e. it holds
	\begin{align*}
	\mathbb{E}[\|w^0\|_{H^1(\mathbb{S}^2)}^2]=\mathbb{E}[\|w\|_{H^1(\mathbb{S}^2)}^2]\leq (\|\partial_x h\|^2_{L^2}+\bar{G}^2C(\lambda_1,\lambda_2))=:K\, ,
	\end{align*}
	for every stationary solution $w$.	There exists a constant $C>0$ common to each initial condition so that $\mathbb{E}[\|w^0\|^{1/2}_{H^2(\mathbb{S}^2)}]<C$. 
	Moreover, each invariant measure is supported on $H^2(\mathbb{S}^2)$, i.e. $H^2(\mathbb{S}^2) \subseteq  \mathrm{supp}\mu $.
\end{theorem}
\begin{proof}
	In order to consider $\mu\in\mathcal{I}$ as distribution for an initial condition $w^0$ to \eqref{LLG}, we look at its regularity. To this aim, we need $w^0\in \mathcal{L}^2(\Omega;H^1(\mathbb{S}^2))$. We look therefore at the second moment of the $H^1(\mathbb{S}^2)$ of $w^0$, 
	\begin{align*}
	\mathbb{E}[\|w^0\|^2_{H^1(\mathbb{S}^2)}]=\int_{\Omega} \|w^0(\omega)\|^2_{H^1(\mathbb{S}^2)}\dd \mathbb{P}(\omega)=\int_{H^1(\mathbb{S}^2)} \|v\|^2_{H^1(\mathbb{S}^2)}\dd \mu(v)\, ,
	\end{align*}
	which needs to be finite.  Let $(\mu_{t_{n_k}})_{n_k}$ the subsequence converging weakly to $\mu$, then from Fatou's Lemma 
	\begin{align}\label{eq:regol_fatou}
	\int_{H^1(\mathbb{S}^2)} \|v\|^2_{H^1(\mathbb{S}^2)}\dd \mu(v)\leq \liminf_{R\rightarrow +\infty} \lim_{k\rightarrow +\infty} \int_{H^1(\mathbb{S}^2)} \|v\|^2_{H^1(\mathbb{S}^2)}\wedge R\dd \mu_{t_{n_k}}(v)\, ,
	\end{align}
	where $\wedge$ denotes the minimum of the two quantities (in particular the integrand is a continuous and bounded function from $H^1(\mathbb{S}^2)$ with real values). By the definition of $\mu_{t_{n_k}}$, it follows that
	\begin{align}\label{eq:regol_def_mu_k}
	\int_{H^1(\mathbb{S}^2)} \|v\|^2_{H^1(\mathbb{S}^2)}\wedge R\dd \mu_{t_{n_k}}(v)=\frac{1}{t_{n_k}}\int_{0}^{t_{n_k}} \int_{H^1(\mathbb{S}^2)} \|v\|^2_{H^1(\mathbb{S}^2)}\wedge R\dd (\mathbb{P}\circ (u^{u^0}_r)^{-1} )(v)\dd r\, .
	\end{align}
	We recognise on the right hand side of \eqref{eq:regol_def_mu_k} the definition of expectation of $\|u_r^{u^0}\|_{H^1(\mathbb{S}^2)}\wedge R$ with respect to $\mathbb{P}$, which leads from the monotone convergence theorem and Lemma \ref{lemma:lemma_sup_grad} to	
	\begin{align*}
	\mathbb{E}[\|w^0\|^2_{H^1(\mathbb{S}^2)}] \leq \frac{1}{t_{n_k}}\int_{0}^{t_{n_k}} \mathbb{E}[ \|u^{u^0}_r\|^2_{H^1(\mathbb{S}^2)}]\dd r\lesssim \frac{\mathbb{E}[\|u^0\|^2_{H^1(\mathbb{S}^2)}]+t_{n_k}\|\partial_x h\|^2_{L^2}}{t_{n_k}}\, ,
	\end{align*}
	which is bounded. We have concluded that we can choose as initial condition to \eqref{LLG} any random variable $w^0$ distributed like an invariant measure $\mu\in \mathcal{I}$, since $w^0\in \mathcal{L}^2(\Omega;H^1(\mathbb{S}^2))$. This implies existence of stationary solutions to \eqref{LLG}. An analogous procedure and Lemma \ref{eq:estimate_laplaciano_time} leads to
	\begin{align*}
	\mathbb{E}[\|w^0\|^{1/2}_{H^2(\mathbb{S}^2)}]=\int_{H^2(\mathbb{S}^2)}\|v\|^{1/2}\dd \mu(v) <C\, ,
	\end{align*}
	which allows us to conclude that $ H^2(\mathbb{S}^2) \subseteq  \mathrm{supp}\,\mu$. 	If $w^0\in \mathcal{L}^2(\Omega;H^1(\mathbb{S}^2))$, the associated stationary solution $w$ satisfies the energy inequality \eqref{eq:unif_bound_time_1}. From the stationarity of $w$, we deduce that
	\begin{align*}
	\mathbb{E}[\|\nabla w^0 \|_{L^2}^2]\leq (\|\partial_x h\|^2_{L^2}+\bar{G}^2C(\lambda_1,\lambda_2))\, .
	\end{align*}
	This implies that every stationary solution is bounded by the same constant $K>0$.
\end{proof}
\begin{remark}
Also assuming $u^0\in \mathcal{L}^4(\Omega;H^1(\mathbb{S}^2))$, it is not possible to achieve $w^0\in \mathcal{L}^4(\Omega, H^1(\mathbb{S}^2))$. This is reflected from the estimate in Lemma \ref{lemma:linear_growth_tightness}. As already mentioned, it does not seem possible to sample initial conditions from the invariant measure for martingale solutions. Namely, we can not employ initial conditions distributed like invariant measures $w^0$ as initial conditions for martingale solutions. We can nevertheless interpret the stationary solutions as pathwise solutions, were no integrability requirement for the initial condition is needed.
\end{remark}

\subsection{Existence of ergodic measures}
Denote by $\mathcal{I}$ the set of all invariant measures for $(P_t)_t$ on $(H^1(\mathbb{S}^2),\mathcal{B}_{H^1(\mathbb{S}^2)})$. We address the problem of existence of ergodic invariant measures in $\mathcal{I}$. Denote by $L^2(H^1(\mathbb{S}^2),\mu)$ the space of maps $\phi:H^1(\mathbb{S}^2)\rightarrow\mathbb{R}$ so that the expectation with respect to $\mu$ of the second moment of $\phi$ is bounded. Recall that a measure $\mu$ is ergodic provided for all $\phi\in L^2(H^1(\mathbb{S}^2),\mu)$
\begin{align*}
\lim_{T\rightarrow +\infty} \frac{1}{T}\int_{0}^{T}P_t\phi \dd t=\int_{H^1(\mathbb{S}^2)} \phi(v)\dd \mu(v) \, .
\end{align*}
We prove existence of an ergodic invariant measure in the set $\mathcal{I}$ by means of two well known results in Proposition \ref{prop:inv_measure_extr_points} and Theorem \ref{teo:Krein-Milman}. 
\begin{proposition} (e.g. Proposition 3.2.7 in \cite{DaPrato_Zap_inv} )\label{prop:inv_measure_extr_points}
	An invariant measure for the semigroup $(P_t)_t$ is ergodic if and only if it is an extremal point of the set  $\mathcal{I}$.
\end{proposition}
Recall some basic definitions. Consider a convex subset $K$ of a Hausdorff topological vector space. The extreme points of the set $K$ are elements of the set which do not lie in any open line segment joining two points of $K$. The convex hull of $K$ is the smallest convex set containing $K$. We now state the Krein-Milman theorem.
\begin{theorem}\label{teo:Krein-Milman} (Krein-Milman Theorem,  \cite{Krein_Milman})
Any compact convex subset $K$ of a Hausdorff locally convex topological vector space is equal to the closed convex hull of its extreme points. 
\end{theorem}
We approach the main result of this section.  We assume again to work on a probability space $(\Omega,\mathcal{F},(\mathcal{F}_t)_t,\mathbb{P})$ where the Brownian motion $W$, the initial condition $u^0$ and the initial condition $w^0$ are adapted (which can be constructed via Skorohod representation theorem).

\begin{theorem}\label{teo:beta_neq_0}
	Assume $u^0\in \mathcal{L}^2(\Omega;H^1(\mathbb{S}^2))$. Then, there exists at least an ergodic invariant measure for the semigroup $(P_t)_t$ associated to \eqref{LLG}. 
\end{theorem}
\begin{proof}

	 From Theorem \ref{teo:existence_invariant_measure}, $\mathcal{I}\neq \emptyset$. We aim to apply the Krein-Milman Theorem \ref{teo:Krein-Milman} and prove that the set of extremal points of $\mathcal{I}$ is not empty. This implies existence at least an ergodic invariant measure from Proposition \ref{prop:inv_measure_extr_points} and concludes the proof.
	 
	In our framework, the Hausdorff locally convex topological vector space is $\mathcal{P}(H^1(\mathbb{S}^2))$ equipped with the topology of the weak convergence and $\mathcal{I}$ corresponds to $K$. 
	
	Indeed, the set $\mathcal{I}$ is convex in $\mathcal{P}(H^1(\mathbb{S}^2))$. Indeed for all $\alpha\in[0,1]$ and for all $\mu,\nu\in \mathcal{I}$ the sum of the measures $\alpha\mu+(1-\alpha)\nu$ is also probability measure on $H^1(\mathbb{S}^2)$. The probability measure $\alpha\mu+(1-\alpha)\nu$ is also an element of $\mathcal{I}$, indeed from the linearity of $(P^*_t)_t$ and the invariance of $\mu,\nu$ it follows that
	\begin{align*}
	\int_{E}\phi \dd \left(P^*_t(\alpha \mu+(1-\alpha)\nu)\right)=\alpha\int_{E}\phi \dd \left(P^*_t\mu\right)+(1-\alpha)\int_{E}\phi \dd \left(P^*_t\nu\right)=
	\alpha\int_{E}\phi \dd \mu+(1-\alpha)\int_{E}\phi \dd \nu\, ,
	\end{align*}
	which shows that  $\alpha\mu+(1-\alpha)\nu\in\mathcal{I}$. 
	
	We show now that the set of invariant measures $\mathcal{I}$ is compact. As a consequence of the Feller's property, $\mathcal{I}$ is closed: therefore we just need to prove that $\mathcal{I}$ is precompact. This reduces to show that $\mathcal{I}$ is tight in $\mathcal{P}(H^1(\mathbb{S}^2))$.  Let $\mu\in \mathcal{I}$ and consider for $R>0$ the measurable set $\{\|x\|_{H^2}>R\}\subset\mathcal{B}(H^1(\mathbb{S}^2))$. It follows from the invariance of $\mu$ and from the a priori bounds for the Krylov-Bogoliubov theorem (applied to the stationary solutions). Consider an initial condition $w^0$ adapted on the joint probability space given by $w^0$ and the Brownian motion $W$, so that $w^0$ is distributed like $\mu$ (via Skorohod).  From Theorem \ref{teo:pathwise_stationary_sol}, there exists a pathwise stationary solution.  We consider
	\begin{align*}
	\mu(\{\|x\|_{H^2}>R\})=\int_{H^1(\mathbb{S}^2)} \mathbf{1}_{\{\|x\|_{H^2}>R\}}(v)\dd \mu(v)=\int_{\Omega}\mathbf{1}_{\{\|w^0(\omega)\|_{H^2}>R\}}(v)\dd  \mathbb{P}(\omega)=\mathbb{P}(\{\|w^0\|_{H^2}>R\})\, .
	\end{align*}
	From Theorem \ref{teo:regolarity_pathwise_stationary_sol}, $w^0\in \mathcal{L}^2(\Omega;H^1(\mathbb{S}^2))$ and there exists a common constant $C>0$ such that $\mathbb{E}[\|w^0\|_{H^1}^2]<C$, for any given initial condition distributed like an invariant measure. By employing the stationarity of the solution, Lemma \ref{lemma:linear_growth_tightness} and the uniform bound in Theorem \ref{teo:regolarity_pathwise_stationary_sol}
	\begin{align*}
		\mathbb{P}(\{\|w^0\|_{H^2}>R\})&=\frac{1}{T}\int_{0}^{T}\mathbb{P}(\|w^0\|_{H^2}>R)\dd t=\frac{1}{T}\int_{0}^{T}\mathbb{P}(\|u_t^{w^0}\|^{1/2}_{H^2}>R^{1/2})\dd t\\
		&\leq \frac{1}{T R^{1/2}}\int_{0}^{T}\mathbb{E}[\|u_t^{w^0}\|^{1/2}_{H^2}]\dd t\leq \frac{C(\|\partial_x h\|^2_{L^\infty}+\mathbb{E}[\|w^0\|^2_{H^1}])}{R^{1/2}}\, ,
	\end{align*}
	which shows tightness of the set $\mathcal{I}$. We are therefore in the situation of Krein-Milman theorem and it follows that $\mathcal{I}$ has extremal points. In conclusion, the semigroup $(P_t)_t$ admits at least an ergodic measure.
                     	
\end{proof}
\begin{remark}
The choice of the definition of pathwise solution is determinant: indeed the tightness of the set $\mathcal{I}$ relies on the boundedness of $\mathbb{E}[\|u_t^{w^0}\|^{1/2}_{H^2}]$ from Theorem \ref{teo:regolarity_pathwise_stationary_sol}.
\end{remark}

\section{On uniqueness and non uniqueness of stationary solutions and invariant measures.}\label{sec:uniq_stat_sol_inv_measure}
Throughout this section, we assume the noise to has either shape \eqref{eq:int_a} or shape \eqref{eq:int_b}, where \\
\noindent
\begin{tabularx}{\linewidth}{XX}
	\begin{equation}
	\int_{0}^{t} u_r\times h_1 \circ \dd B_r\label{eq:int_a} \, ,
	\end{equation}
	& 
	\begin{equation}
	\quad\int_{0}^{t} h_2 u_r\times  \circ \dd \bar{B}_r\label{eq:int_b} \, ,
	\end{equation}
\end{tabularx}
$h_1\in H^1(D;\mathbb{R}^3)$, $B$ is a real valued Brownian motion, $h_2\in H^1(D;\mathbb{R})$ and $\bar{B}$ is a $\mathbb{R}^3$-valued Brownian motion.
We make further assumptions on $h_1,h_2$ to determine explicitly stationary solutions to \eqref{LLG_intro}, namely
\begin{tabularx}{\linewidth}{XX}
	\begin{equation}
	h_1\neq 0,\quad \partial_x h_1 =0\label{eq:H_a} \, ,
	\end{equation}
	& 
	\begin{equation}
	\quad h_2\neq 0,\quad \partial_x h_2 =0\label{eq:H_b} \, .
	\end{equation}
\end{tabularx}
Note that if $h_1=0$ (resp, $h_2=0$), it is known that static solutions to \eqref{LLG_intro} exist and are not unique. The associated invariant measures are Dirac measures centred in those static solutions. Under \eqref{eq:H_a}, the choice of the noise \eqref{eq:int_a} leads to non uniqueness of stationary solutions for some particular values of $h_1$. Instead, by employing noise \eqref{eq:int_b} under \eqref{eq:H_b} there exists a unique stationary solution. Define for all the maps $v:[0,T]\rightarrow\mathbb{R}^3$ the drift 
\begin{align*}
\mathcal{D}(v)_{s,t}:=\int_{s}^{t}\left[\lambda_1 v_r\times g'(v_r)-\lambda_2v_r\times (v_r\times g'(v_r))\right]\dd r\, .
\end{align*}
We state the following additional condition on the anisotropy, which is a smallness condition on the coefficients of the matrix and the additive vector $b$.
\begin{assumption}\label{assumption:anisotropic_small}
	Assume that the anisotropy is of the form $g'(x):=Ax+b$, where $A$ is a real valued $3\times 3$ matrix and $b\in \mathbb{R}^3$.  Assume further that $\bar{G}:=2\sup_{i,j}|A_{i,j}|^2+|b|<\lambda_2/2C_p(2\lambda_2+|\lambda_1|)$, where $C_p$ is the Poincaré's constant associated to the domain $D$.
\end{assumption}
Under Assumption \ref{assumption:anisotropic_small}, we can achieve an improved a priori bound, which we state. Without loss of generality, we prove the result only for the noise \eqref{eq:int_a} (see the discussion in Remark \ref{ref:different_choice_noise}).
\begin{proposition}\label{pro:improved_estimate_anisotropy}
Consider a solution to the stochastic LLG equation driven by the noise \eqref{eq:int_a}.
Under Assumption \ref{assumption:anisotropic_small}, every solution to the stochastic LLG equation satisfies 
\begin{equation}\label{eq:unif_bound_time_anisotrpoic}
\begin{aligned}
\sup_{r\in[0,t]}\mathbb{E}\left[\|\partial_x u_r\|^2_{L^2}\right]+C(\lambda_1,\lambda_2)\int_{0}^{t}\mathbb{E}\left[\|u_r\times\partial_x u_r\|^2_{L^2}\right]\dd r+&\lambda_2\int_{0}^{t}\mathbb{E}\left[\|u_r\times\partial^2_x u_r\|^2_{L^2}\right]\dd r\\
&\quad\quad\quad
\leq\mathbb{E}\left[\|\partial_x u^0\|_{L^2}^2\right]+t\|\partial_x h_1\|^2_{L^2}\, .
\end{aligned}
\end{equation}
\end{proposition}
\begin{proof}
We look again at inequality \eqref{eq:unif_bound_time_1}. We ask whether it is possible to obtain an estimate where the linear growth depends only on the derivative of the spatial component of the noise. We look again at the drift
\begin{align*}
2\lambda_1\int_{0}^{T}\int_{D}\partial_x(u_r\times (Au_r+b))\cdot \partial_x u_r \dd x \dd r=2\lambda_1\int_{0}^{T}\int_{D} u_r\times A\partial_xu_r\cdot \partial_x u_r\dd x \dd r\\
\quad\leq 2|\lambda_1|\sup_{i,j}|A_{i,j}|^2\int_{0}^{T}\|\partial_x u_r\|^2_{L^2}\dd r\, .
\end{align*}
Now, by recalling that $a\times (b\times c)=b (c\cdot a)-c(a\cdot b)$ for all $a,b,c\in \mathbb{R}^3$, we observe that
\begin{align*}
\partial_x (u\times (u\times g'(u)))\cdot \partial_x u&=u\times (\partial_x u\times g'(u))\cdot \partial_x u+u\times (u\times A\partial_x u)\cdot \partial_x u\\
&= [\partial_x u (u\cdot g'(u))-g'(u)(\partial_x u\cdot u)]\cdot \partial_x u+u\times (u\times A\partial_x u)\cdot \partial_x u\\
&=|\partial_x u|^2(u\cdot g'(u))-A\partial_x u\cdot \partial_x u\, ,
\end{align*}
where in the last equality we used, from the same vectorial relation, that
\begin{align*}
[u\times(u\times A\partial_x u)]\cdot \partial_x u=[u (A\partial_x u\cdot u)-A\partial_x u]\cdot \partial_x u=-A\partial_x u\cdot \partial_x u\, .
\end{align*}
This implies the estimate 
\begin{align*}
2\lambda_2\int_{0}^{T}\int_{D}\partial_x (u_r\times (u_r\times g'(u_r)))\cdot \partial_x u_r \dd x \dd r\leq 4\lambda_2 \bar{G}\int_{0}^{T}\|\partial_x u_r\|^2_{L^2}\dd r\, .
\end{align*}
By following the steps of Lemma \ref{lemma:lemma_sup_grad} and by applying Lemma \ref{lemma:poincarre}, we obtain the inequality
\begin{equation}
\begin{aligned}
\sup_{r\in[0,T]}\mathbb{E}&\left[\|\partial_x u_r\|^2_{L^2}\right]+\lambda_2C_p^{-1}\int_{0}^{T}\mathbb{E}\left[\|u_r\times\partial_x u_r\|^2_{L^2}\right]\dd r+\lambda_2\int_{0}^{t}\mathbb{E}\left[\|u_r\times\partial^2_x u_r\|^2_{L^2}\right]\dd r\\
&\quad\quad\quad\quad\leq\mathbb{E}\left[\|\partial_x u^0\|_{L^2}^2\right]+t\|\partial_x h_1\|^2_{L^2}+[4\lambda_2 \bar{G}+2|\lambda_1 \sup_{i,j}|A_{i,j}|^2|]\int_{0}^{T}\|\partial_x u_r\|^2_{L^2}\dd r\, .
\end{aligned}
\end{equation}
From Lemma \ref{lemma:u_dot_nablau_0}, the equality $\|\partial_x u_r\|^2_{L^2}=\|u_r\times\partial_x u_r\|^2_{L^2}$ holds. By assuming $\bar{G}<\lambda_2/2C_p(2\lambda_2+|\lambda_1|)$, inequality \eqref{eq:unif_bound_time_anisotrpoic} follows.
\end{proof}
\begin{remark}
	The difference between \eqref{eq:unif_bound_time_1} and \eqref{eq:unif_bound_time_anisotrpoic} is that the second inequality the right hand side depends on time only through $\partial_x h_1$. Hence, the linear growth in \eqref{eq:unif_bound_time_1}  is the only one needed to prove existence of an invariant measure: there is no restriction on the anisotropy. The second estimate in \eqref{eq:unif_bound_time_anisotrpoic} is needed for the long time behaviour.
\end{remark}
\begin{remark}\label{remark:partial_x_h_0}
	We observe that, if $\partial_x h_1=0$, then the estimate \eqref{eq:unif_bound_time} is uniform in time. By looking at the proof of estimate \eqref{eq:unif_bound_time}, we realize that the estimate is actually pathwise (due to the orthogonality of the noise) and reduces to
	\begin{equation}\label{eq:unif_bound_time_remark}
	\begin{aligned}
	\lambda_2\int_{0}^{t}\|u_r\times\partial^2_x u_r\|^2_{L^2}\dd r\leq\|\partial_x u^0\|_{L^2}^2\, ,
	\end{aligned}
	\end{equation}
	namely we have a uniform in time estimate, which leads to
	\begin{align*}
	\int_{0}^{+\infty}\|u_r\times\partial^2_x u_r\|^2_{L^2}\dd r<+\infty\; ,\quad\int_{0}^{+\infty}\|\partial_x u_r\|^2_{L^2}\dd r<+\infty\, .
	\end{align*}
\end{remark}
We state the main result of this section: we describe explicitly the shape of the stationary solutions under the above conditions on the noise and on the smallness condition on the anisotropy. 
\begin{theorem}\label{teo:sde_h_1}
Assume that \eqref{LLG} is driven by the noise \eqref{eq:int_a}, under condition  \eqref{eq:H_a}, and assume Assumption \ref{assumption:anisotropic_small} on the anisotropy. Then every stationary solution to \eqref{LLG} is a solution to
\begin{align}\label{eq:SDE_h_1}
\delta w_{s,t}=\mathcal{D}(w)_{s,t}+\int_{s}^{t}w_r\times h_1 \circ \dd B_r\, .
\end{align}
For every fixed $h_1$, there exist more solutions to \eqref{eq:SDE_h_1} and thus more invariant measures $\mu$ to \eqref{LLG}. 
\end{theorem}
\begin{remark}
	One can check that, for instance, that $+h_1$ and $-h_1$ are solutions to \eqref{eq:SDE_h_1}, provided $|h_1|=1$ (this is without lost of generality, since it is enough to normalize $h_1$). The invariant measures to the system are Dirac deltas centred in $h_1$, $-h_1$. There are more invariant measures, as proved in \cite{Neklyudov_Prohl} (see also Proposition 1.17 in \cite{Banas_book}). This situation can be avoided by taking a sum of different independent Brownian motion and associated vectors $h_2$ (more than two).
\end{remark}
\begin{proof}
We aim to show that every stationary solution to \eqref{LLG}, under the current hypothesis on the noise, is a solution to \eqref{eq:SDE_h_1}. Let $\mu$ be an invariant measure and consider $w^0\in \mathcal{L}^2(\Omega;H^1(\mathbb{S}^2))$ as an initial condition to \eqref{LLG}.  Denote by $w$ the solution to \eqref{LLG} started in $w^0$. 
Since $w$ is stationary, then for all $t>0$
\begin{align*}
\mathbb{E}[\|w_t\|^2_{H^1}]=\mathbb{E}[\|w^0\|_{H^1}^2]\, .
\end{align*}
The stationary solution $w$ has to satisfy the inequality in Proposition \ref{pro:improved_estimate_anisotropy}. Combining these two facts and the condition on the noise $\partial_x h_1=0$, it follows for every $t>0$ that
\begin{align*}
C(\lambda_1,\lambda_2)\int_{0}^{t}\mathbb{E}\left[\|w_r\times\partial_x w_r\|^2_{L^2}\right]\dd r+&\lambda_2\int_{0}^{t}\mathbb{E}\left[\|w_r\times\partial^2_x w_r\|^2_{L^2}\right]\dd r=0\, .
\end{align*}
In particular, for a.e. $(x,t)\in D\times [0,+\infty)$ and $\mathbb{P}$-a.s. ti holds $w_t(x)\times \partial_x w_t(x)=0$.

Since $w$ and $\partial_x w$ are orthogonal and $w_t(x)\in\mathbb{S}^2$, it follows that $\partial_xw\equiv 0$ $\mathbb{P}$-a.s. Analogously it follows $w_t(x)\times \partial_x^2 w(x)=0$ for a.e. $(x,t)\in D\times [0,+\infty)$ and $\mathbb{P}$-a.s. From the orthogonality relation relation, for a.e. $(x,t)\in D\times [0,+\infty)$ and $\mathbb{P}$-a.s.
\begin{align*}
\partial_x^2 w_t(x)=w_t(x)|\partial_x w_t(x)|^2+w_t(x)\times\partial_x^2 w_t(x)\, ,
\end{align*}
it follows that $\partial_x^2 w_t(x)=0$ for a.e. $(x,t)\in D\times [0,+\infty)$ and $\mathbb{P}$-a.s.. As a consequence of the continuity of the derivative (because of the continuous embedding of $H^1$ in the space of continuous functions), we need to conclude that we have to take $\partial_x w^0=0$ on $\bar{D}$. In particular, $w$ is a solution to \eqref{eq:SDE_h_1}. As proved in \cite{Neklyudov_Prohl}, \eqref{eq:SDE_h_1} admits more more solutions for different values of $h_1$, thus more stationary solutions to \eqref{LLG}. In particular, for each fixed $h_1$ there are more Dirac delta measures centred in the  stationary measures.

\end{proof}

In contrast to Theorem \ref{teo:sde_h_1}, when considering noise \eqref{eq:int_b} there exists only a stationary measure and a unique invariant measure.
\begin{theorem}\label{teo:sde_h_2}
	Assume that \eqref{LLG} is driven by the noise \eqref{eq:int_b}, under condition  \eqref{eq:H_b}.  Let Assumption \ref{assumption:anisotropic_small} hold on the anisotropy. Then every stationary solution to \eqref{LLG} is a solution to
	\begin{align}\label{eq:SDE_h_2}
	\delta v_{s,t}=\mathcal{D}(v)_{s,t}+\int_{s}^{t}h_2v_r\times\circ \dd B_r\, .
	\end{align}
	There exists a unique solution $v$ to \eqref{eq:SDE_h_2}. Hence \eqref{LLG} admits as unique stationary solution $v$ and a unique invariant measure $\mu$. Moreover, $\bar{\mu}$ is a Gibbs measure on $(\mathbb{S}^2,\mathcal{B}_{\mathbb{S}^2})$
	\begin{align}\label{eq:Gibbs}
	\bar{\mu}[\dd v]=\frac{\exp(-\frac{\lambda_2}{h_2}\bar{\mathcal{E}}(v))\dd v}{\int_{\mathbb{S}^2}\exp(-\frac{\lambda_2}{h_2}\bar{\mathcal{E}}(z))\dd z}\, ,
	\end{align}
	where $\bar{\mathcal{E}}(v)=\int_{D}g'(v)\cdot v\dd x=|D|g'(v)\cdot v$,   for $v\in \mathbb{S}^2$ (the integral on the sphere is with respect to the Riemmanian volume measure). 
\end{theorem}

\begin{proof}
	By following the lines of Theorem \ref{teo:sde_h_1} with the estimate in Lemma \ref{lemma:linear_growth_tightness}, we can show that every stationary solution to \eqref{LLG} endowed with the noise  \eqref{eq:H_b}, under condition  \eqref{eq:H_b}, fulfils equation \eqref{eq:SDE_h_2}. For every fixed constant $h_2\neq 0$, there exists a unique solution $v$ to \eqref{eq:SDE_h_2} (see Section 1.2.2 in \cite{Banas_book}). Thus \eqref{LLG} admits $v$ as unique stationary solution. Furthermore, \eqref{eq:SDE_h_2} has a unique invariant measure of the form \eqref{eq:Gibbs}, as proved by Neklyudov, Prohl \cite{Neklyudov_Prohl}. The authors prove that the infinitesimal generator of the semigroup associated to \eqref{eq:SDE_h_2} has the form
	\begin{align*}
	\mathcal{A}=\frac{h_2^2}{2}\Delta_{\mathbb{S}^2}-\left[\lambda_1v_r\times g'(v_r)-\lambda_2v_r\times (v_r\times g'(v_r))\right]\cdot \nabla\, ,
	\end{align*}
	and that the unique invariant measure is a Gibbs measure (see Section 1.2, Theorem 1.7 in \cite{Banas_book}).
\end{proof}

\begin{remark}
	In absence of anisotropic energy, \eqref{eq:SDE_h_2} reduces to the so called \enquote{spherical Brownian motion} (see the Appendix for more details on this equation).  A fundamental behaviour of the Brownian motion on the sphere is the existence of a unique probabilistic invariant measure of the form $\mu[\dd v]=\dd v/|\mathbb{S}^2|$, 	namely the uniform distribution (as proved by Van den Berg, Lewis \cite{BM_hypersurface}).  The uniform distribution is trivially a Gibbs measure. 
\end{remark}
\begin{remark}\textit{(On the stochastic LLG in more dimensions)}
Observe that, in absence of anisotropic energy and provided the equation \eqref{LLG} is driven by noise \eqref{eq:int_b} under the conditions \eqref{eq:H_b}, the solution to equation \eqref{eq:SDE_h_2} is also a probabilistically and analytically strong solution to \eqref{LLG} in every space dimensions. The associated measure is, also in this case, a Gibbs measure as above. 
\end{remark}
\begin{remark}\label{Goldys}
	In Theorem \ref{teo:sde_h_2} the invariant measure does not depend on $\lambda_1$. In particular, it is possible to obtain the same results if $\lambda_1=0$. This goes in the direction of looking only at the simplified model with $\lambda_1=0$, since this term does not seem to contribute significantly.
\end{remark}
\begin{remark}\label{remark:magn_rev}
	A second interesting behaviour, already observed by Brze\'zniak, Goldys and Jegaraj \cite{brzezniak_LDP} in a particular case, is the phenomenon of magnetization reversal: the Brownian motion with values on the sphere is a recurrent Markov process (since it is valued on a compact manifold), hence also the solution to $u$ to \eqref{LLG} in the large time exhibits a recurrent behaviour. The Brownian motion $v$ is recurrent and from the continuity of the solution with respect to the initial condition in $L^2$, namely
	\begin{align*}
	\sup_{0\leq t\leq T}\|u_t-v_t\|^2_{L^2}\lesssim \|u^0-v^0\|^2_{L^2}\, .
	\end{align*}
	This implies that if the initial conditions are close, the solutions will be close. In particular the paths of the two solutions are also close.
\end{remark}

\subsection{Long time behaviour if $\partial_x h_2=0$ (resp $\partial_x h_1=0$).}\label{sec:large_time_behaviour}
We address the so called \enquote{long time behaviour} problem for \eqref{LLG}: do the global solutions solution $u$ to \eqref{LLG} started in $u^0\in H^1(\mathbb{S}^2)$ converge to the unique stationary solution $v$ started in an initial condition $v^0$ distributed like the invariant measure for long times? 
Provided $\partial_x h_2=0$(resp $\partial_x h_1=0$) and $g\equiv 0$, we can say that the solutions to \eqref{LLG} started in $u^0$ converges to its average $\langle u  \rangle$ \enquote{at time $+\infty$}. This space average is itself a solution to the equation satisfied by the stationary solutions: hence every solution, for big times, converges to a stationary solution.
\subsection{Some heuristics on the long time behaviour also in presence of anisotropic energy.}\label{sec:heuristics_Long_time_behaviour}
In this section, we derive an intuition on the limit behaviour of the trajectories of the stochastic LLG. In this Section \ref{sec:heuristics_Long_time_behaviour}, we assume that the anisotropic energy is non zero, i.e. $g\neq 0$ and satisfied Assumption \ref{assumption:anisotropic_small}. From Remark \ref{remark:partial_x_h_0}, the integral
\begin{align*}
\int_{0}^{+\infty}\|u_r\times\partial_x u_r\|^2_{L^2}\dd r<+\infty\, 
\end{align*}
is bounded; this implies that there exists a monotonic increasing subsequence $(t_k)_k\subset[0,+\infty)$ such that $\mathbb{P}$-a.s.
\begin{align}\label{eq:conv_stat_1}
\lim_{k\rightarrow +\infty}\|u_{t_k}\times\partial_x u_{t_k}\|^2_{L^2} =\lim_{k\rightarrow +\infty}\|\partial_x u_{t_k}\|^2_{L^2}=0\, ,
\end{align}
where the last equality follows from Lemma \ref{lemma:u_dot_nablau_0}. From the pathwise energy inequality \eqref{eq:unif_bound_time_remark}, we observe that
\begin{align*}
\sup_{t\geq T} \|\partial_x u_t\|^2_{L^2}\lesssim \|\partial_x u_T\|^2_{L^2}\lesssim \|\partial_x u_{t_k}\|^2_{L^2}\, ,
\end{align*}
$\mathbb{P}$-a.s. for $t_k<T$. From Poincaré-Wirtinger inequality (see Theorem \ref{teo:Poin_wirt}), each solution fulfilling \eqref{eq:unif_bound_time_remark} satisfies for big times
\begin{align}\label{eq:conv_stat_2}
\sup_{t\geq T}\|u_t-\langle u_t\rangle\|^2_{L^2}\leq C_p \sup_{t\geq T}\|\partial_x u_{t}\|^2_{L^2}\lesssim\|\partial_x u_{t_k}\|^2_{L^2}\, ,
\end{align}
$\mathbb{P}$-a.s for $t_k<T$ and where $\langle u_t\rangle$ is the spatial average. Thus every solution $u$ converges for big times to its mean value $\langle u_{\infty} \rangle$ (if it exists).

Intuitively, since $u$ for big times converges to its limit spatial average, the limit spatial average itself $U:=\langle u_{\infty} \rangle=\lim_{t\rightarrow +\infty}\langle u_t\rangle $ needs to be a solution to the stochastic LLG. The terms in the drift where a derivative appear need to vanish and $U$ is the unique solution to
\begin{align}\label{eq:limit_eq_average}
\delta U_{s,t}=\int_{s}^{t} [U_r\times g'(U_r)-U_r\times (U_r\times g'(U_r))]\dd r+\int_{s}^{t}h_2 U_r\times \circ \dd W_r\, .
\end{align}
If the last part can be made rigorous, it proves that in the long time every trajectory to the stochastic LLG converges to a stationary solution (indeed recall that \eqref{eq:limit_eq_average} coincides with the equation fulfilled by the stationary solutions in Theorem \ref{teo:sde_h_2}).

\subsection{A rigorous proof in case of null anisotropic energy.}
We assume now that there is no anisotropic energy in the system, i.e. $g\equiv 0$. 
The following Proposition \ref{prop:long_time_behaviour} shows rigorously that the limit behaviour of each solution $u$ is indeed a SDE (no spatial dependence in the limit), where the derivative terms vanish. In particular, we conclude that $u$ differs from $B$ up to a constant vector for large times.
\begin{proposition}\label{prop:long_time_behaviour}
	(Long time behaviour) 
	Under the condition $h_2\neq0$, $\partial_x h_2=0$ (resp. $h_2\neq0$, $\partial_x h_2=0$) and for $g\equiv0$, there exists a random variable $\alpha$ such that
	\begin{align*}
	\lim_{T\rightarrow+\infty}\sup_{t\geq T}\|(u_t-B_t)\cdot(u_t-B_t)-\alpha\|_{L^1}^2=0\, .
	\end{align*}
	In particular, for large times the solution $u$ converges $\mathbb{P}$-a.s. to a Brownian motion $B$ with values on the sphere up to a constant, i.e. for $T\rightarrow +\infty$ for $t>0$
	\begin{align*}
    |u_{T+t}-B_{T+t}|^2-|u_{T}-B_T|^2=0
	\end{align*}
	in $C_b([0,+\infty);L^1)$.
\end{proposition}
\begin{proof} We follow the strategy of Theorem 3.6 in \cite{Dabrock_Hofmanova_Roger}. The proof below works with small adaptations for the noise \eqref{eq:int_a} under the assumptions \eqref{eq:H_a}.
	We want to prove the convergence of $u$ to $B$ for large times, therefore we look at the equation of the difference $u-B$, which reads in the Stratonovich formulation (respectively in the It\^o formulation)
	\begin{align*}
	\delta(u-B)_{s,t}&=\int_{s}^{t}b(u_r)\dd r+\int_{s}^{t}h_2\big[u_r-B_r \big]\times\circ\dd W_r\\
	&=\int_{s}^{t}b(u_r)\dd r+\int_{s}^{t}h_2^2\big[u_r-B_r \big]\dd r+\int_{s}^{t}h_2\big[u_r-B_r \big]\times\dd W_r\, .
	\end{align*}
	Consider the squared equation (from the It\^o's formula on the second formulation of the equation as in Lemma \ref{lemma:linear_growth_tightness}, where $\partial_x h_2$ is assumed to be $0$)
	\begin{align*}
	\delta \int_{D}(u-B)\cdot(u-B)_{s,t}\dd x&=\int_{s}^{t}\int_{D}b(u_r)\cdot (u_r-B_r)\dd x\dd r\, .
	\end{align*}
	We introduce now a random variable $\alpha$, which is independent on time and space
	\begin{align}
	\alpha:= \frac{1}{|D|}\int_{D}(u^0-B^0)\cdot(u^0-B^0)\dd x+\int_{0}^{+\infty} \frac{1}{|D|}\int_{D} b(u_r)\cdot(u_r-B_r)\dd x\dd r\, .
	\end{align}
	We prove that $|\alpha| <+\infty$. Recall the equivalent formulation of the drift $b(u_r)=u_r\times \partial_x u_r+\partial_x^2 u_2 +u_r|\partial_x u_r|^2$.
	By integrating by parts, the fact that $\partial_x B_r=0$ for all $x\in D$, from Stokes theorem and from the null Neumann boundary conditions,
	\begin{align*}
		&\quad\int_{D} [u_r\times \partial_x^2  u_r +\partial_x^2   u_r]\cdot B_r \dd x=\int_{D} u_r\times \partial_x^2  u_r \cdot B_r\dd x+\int_{D}\partial_x^2   u_r\cdot B_r \dd x\\ 
		&=-\int_{D} u_r\times \partial_x  u_r \cdot \partial_x B_r\dd x+\int_{\partial D}\frac{\partial u(y)}{\partial \mathbf{n}(y)}\cdot B\times u(y)\dd \sigma(y)\\
		&\quad - \int_{D}\partial_x u_r \cdot \partial_x B_r\dd x+\int_{\partial D}\frac{\partial u(y)}{\partial \mathbf{n}(y)}\cdot B \dd \sigma(y)= 0\, .
	\end{align*}
	Hence, we are left with
	\begin{align*}
	\left|\frac{1}{|D|}\int_{0}^{+\infty}	\int_{D} u_r|\partial_x u_r|^2\cdot B_r \dd x\right|\leq \frac{1}{|D|}\int_{0}^{+\infty}\|\partial_x u_r\|^2_{L^2} \dd r  \leq \frac{C_p}{|D|}\int_{0}^{+\infty}\|u_r\times\partial^2_x u_r\|^2_{L^2} \dd r \leq \frac{C_p}{|D|}\|\partial_x u^0\|^2_{L^2}\, .
	\end{align*}
	In conclusion, the random variable $\alpha$ fulfils $\mathbb{P}$-a.s.
	\begin{align*}
	|\alpha|\leq 4+\frac{C_p}{|D|}\|\partial_x u^0\|^2_{L^2}\, .
	\end{align*}
	Note that, with the same computations, it holds that $\mathbb{E}[|\alpha|]<+\infty$.
	The expression $(u_t-B_t)\cdot (u_t-B_t)-\alpha$ is meant to look at the difference $u_t-B_t$ for times bigger that $t$: we write
	\begin{align*}
	\left|\int_{D} [(u_t-B_t)\cdot(u_t-B_t)-\alpha]\dd x\right|=&\bigg|-\int_{t}^{+\infty} \int_{D} b(u_r)\cdot(u_r-B_r)\dd x \dd r\bigg|\, .
	\end{align*}
	Notice that it holds that
	\begin{align}\label{eq:lim_pezzo_2}
	\lim_{T\rightarrow+\infty} \mathbb{E}\left[\sup_{t\geq T} \left|\int_{D} [(u_t-B_t)\cdot (u_t-B_t)-\alpha]\dd x\right|\right]=0\, .
	\end{align}
	From the triangular inequality, we obtain
	\begin{align*}
	\|(u_t-B_t)\cdot (u_t-B_t)-\alpha\|_{L^1}&\leq \bigg\|(u_t-B_t)\cdot (u_t-B_t)-\alpha-\frac{1}{|D|}\int_{D} [(u_t-B_t)\cdot (u_t-B_t)-\alpha]\dd x\bigg\|_{L^1}\\
	&+\left|\int_{D} [(u_t-B_t)\cdot (u_t-B_t)-\alpha]\dd x\right|\, .
	\end{align*}
	By applying the Poincaré-Wirtinger inequality, we conclude that
	\begin{align*}
	\bigg\|(u_t-B_t)\cdot (u_t-B_t)-\alpha-\frac{1}{|D|}\int_{D} [(u_t-B_t)\cdot (u_t-B_t)-\alpha]\dd x\bigg\|_{L^1}&\leq C_p \|\partial_x[(u_t-B_t)\cdot (u_t-B_t)]\|_{L^1}\\
	&\leq C_pC\|\partial_x u_t\|_{L^2} \, ,
	\end{align*}
	where $C_p>0$ is the Poincaré-Wirtinger constant and we used that  $B$ and $\alpha$ are constant in space in the last inequality. Thus, by taking the supremum for big times, it follows that
	\begin{align*}
	\sup_{t\geq T}\|(u_t-B_t)\cdot (u_t-B_t)-\alpha\|_{L^1} \lesssim \sup_{t\geq T}\|\partial_x u_t\|_{L^1}+\sup_{t\geq T}\left|\int_{D} [(u_t-B_t)\cdot (u_t-B_t)-\alpha]\dd x\right|\, .
	\end{align*}
	By passing to the limit for $T\rightarrow +\infty$ in the left and in the right hand side of the inequality, the average on the right hand side tends to $0$ as a consequence of \eqref{eq:lim_pezzo_2}. From the computations in Lemma \ref{lemma:lemma_sup_grad}, the bound holds
	\begin{align*}
	\int_{0}^{+\infty}\|\partial_x u_r\|^2_{L^2}\dd r<+\infty\, ,
	\end{align*}
	hence there exists a divergent sequence $(t_k)_k$ such that 
	\begin{align}
	\lim_{k\rightarrow +\infty}\|\partial_x u_{t_k}\|^2_{L^2}=0\, .
	\end{align}
	From the inequalities in Lemma \ref{lemma:lemma_sup_grad}, it follows that
	\begin{align*}
	\sup_{t\geq T} \|\partial_x u_t\|_{L^2}\leq \|\partial_x u_T\|_{L^2}\leq \|\partial_x u_{t_k}\|_{L^2}\, ,
	\end{align*}
	for $t_k<T$. Hence we conclude that
	\begin{align*}
	\lim_{T\rightarrow +\infty}\sup_{t\geq T} \|(u_t-B_t)\cdot (u_t-B_t)-\alpha\|_{L^1} =0\, .
	\end{align*}
	Observe also that for all $t\geq 0$
	\begin{align*}
		\|u_{T+t}-B_{T+t}\|_{L^2}^2&=\int_{D}[u_{T+t}-B_{T+t}]\cdot [u_{T+t}-B_{T+t}]\dd x\\
		&=\int_{D} [u_{T}-B_{T}]\cdot [u_{T}-B_{T}]\dd x+\int_{T}^{T+t}\int_{D} b(u_r)\cdot [u_r-B_r]\dd x \dd r\, .
	\end{align*}
	The conclusion is achieved by observing that
	\begin{align*}
	\sup_{t\geq 0}\||u_{T+t}-B_{T+t}|^2-|u_{T}-B_T|^2\|_{L^1}= \sup_{t\geq T}\|(u_{t}-B_{t})\cdot(u_t-B_t)-\alpha\|_{L^1}\, ,
	\end{align*}
	which converges to $0$ for $T\rightarrow +\infty$ (we used that the domain $D$ is bounded). Observe that the convergence occurs also in expectation, by following the same lines of the above proof. This concludes the proof.
\end{proof}

\begin{remark} (Every solution is synchronised with Brownian motion for big times in case of the noise \eqref{eq:int_b}).
In Remark \ref {remark:magn_rev}, we observe that the solution to the stationary solution and the solutions to \eqref{LLG} under the hypothesis of Proposition \ref{prop:long_time_behaviour} are close. As a consequence of Proposition \ref{prop:long_time_behaviour}, the stationary solution and the other solutions are synchronised for big times. This is due to the fact that the constant $c$ appearing in Proposition \ref{prop:long_time_behaviour} is independent on time and space. One could think that both the solution to \eqref{LLG} and the spherical Brownian motion $B$ are valued on the sphere. Hence their distance is trivially always bounded by a constant. But since the constant is time and space independent, then the motions are synchronised.
\end{remark}
Since the spherical Brownian motion is recurrent, then for big times every solution started in a different initial condition is recurrent on the sphere for large times.
\begin{corollary}
Consider the stochastic LLG driven by the noise \eqref{eq:int_b}, under condition \eqref{eq:H_b}, and $g\equiv 0$.  For large times, every solution to \eqref{LLG} is recurrent.
\end{corollary}

\section{Appendix}

\subsection{Useful results}
We list some useful result from rough path theory. The following classical result (see e.g. \cite{FrizHairer}) enables us to estimate the remainder term.
\begin{lemma}[Sewing lemma]\label{lemma_sewing}
	\label{lemma-lambda}Fix an interval $J$, a Banach space $E$ and a parameter
	$\zeta > 1$. Consider a map $G: I^3 \to E$ such that $ G \in \{\delta H ; \, H: J^2\to E\}$ and for every $s < u < t \in J$,
	\[ |G_{s u t} | \leqslant \omega (s, t)^{\zeta}, \]
	for some regular control $\omega$ on $J$. Then there exists a unique element
	$g \in \mathcal{V}_2^{1 / \zeta} (J; E)$ such that $\delta g = G$
	and for every $s < t \in J$,
	\begin{equation}
	| g_{s t} | \leqslant C(\zeta)\omega (s, t)^{\zeta}
	\label{contraction},
	\end{equation}
	for some universal constant $C_{\zeta}$.
\end{lemma}

We introduce a product formula, which is the equivalent of the Stratonovich product rule in this framework: we employ Proposition 4.1 in \cite{hocquet2018ito} and the modification introduced in \cite{LLG1D}.
	\begin{proposition}[Product formula]
	\label{pro:product}
	Fix an integer $n\ge1$
	and let $a=(a^{i})_{i=1}^n\colon [0,T]\to L^2(D;\R^n)$ (resp.\ $b=(b^i)_{i=1}^n\colon [0,T]\to L^2(D;\R^n)$) be a bounded path, given as a weak solution of the system
	\begin{align*}
		\delta a_{s,t}=\int_{s}^{t} f\dd t+A_{s,t}a_s+\mathbb{A}_{s,t}a_s+a^{\natural}_{s,t}\, , \quad\quad\quad\left(\textrm{resp.}\quad\delta b_{s,t}=\int_{s}^{t} g\dd t+B_{s,t}b_s+\mathbb{B}_{s,t}b_s+b^{\natural}_{s,t}\right)\, ,
	\end{align*}
	
	on $[0,T]\times D,$
	for some $f\in L^2(L^2)$ (resp.\ $g\in L^2(L^2)$).
	We assume that both  
	\[
	\mathbf{A}=\left (A^{i,j}_{s,t}(x),\mathbb
	A^{i,j}_{s,t}(x)\right )_{\substack{1\leq i,j\leq n;\\ s\le t \in [0,T];x\in D }}
	\quad 
	\mathbf{B}=\left (B^{i,j}_{s,t}(x),\mathbb{B}^{i,j}_{s,t}(x)\right )_{\substack{1\leq i,j\leq n;\\ s\le t \in [0,T];x\in D }}
	\]
	are $n$-dimensional geometric rough drivers
	of finite $(p,p/2)$-variation with $p\in[2,3)$ and with coefficients in $H^{1}(D)$.
	Then the following holds:
	\begin{enumerate}[label=(\roman*)]
		\item \label{Q_shift}
		The two parameter mapping $\boldsymbol\Gamma^{\mathbf{A},\mathbf{B}}\equiv(\Gamma^{A,B},\bbGamma^{\mathbf{A},\mathbf{B}})$ defined for $s\le t \in [0,T] $ as
		\begin{equation}
		\label{tensorized_driver}
		\begin{aligned}
		&\Gamma_{s,t}^{A,B}:= A_{s,t}\otimes \mathbf 1 + \mathbf 1\otimes B_{s,t} \, ,\quad 
		\\
		&\bbGamma_{s,t}^{\mathbf{A},\mathbf{B}}:=\mathbb{A}_{s,t}\otimes \mathbf 1+ A_{s,t}\otimes B_{s,t}+ \mathbf 1\otimes \mathbb{B}_{s,t}\, ,
		\end{aligned}
		\end{equation}
		where $\mathbf 1\equiv \mathbf 1_{n\times n}\in \mathcal L(\R^n)$ is the identity, is a $n^2$-dimensional rough driver (in the sense of Definition \ref{defi-RD}),
		\item \label{prod_uv}
		The product $v^{\otimes 2}_t(x)=(a^i_t(x)b^j_t(x))_{1\leq i, j\leq n}$ is bounded as a path in $L^1(D;\R^{n\times n}).$ Moreover, it is a weak solution, in $L^1,$ of the system
		\begin{equation}
		\label{concl:prod}
		d (a\otimes b)= (a\otimes g +f \otimes b )d t + d \boldsymbol\Gamma^{\mathbf{A},\mathbf{B}} [a\otimes b]\,.
		\end{equation}
	\end{enumerate}
\end{proposition}

The following Lemma is a rough path equivalent of the Gronwall's lemma (see \cite{deya2016priori}).
\begin{lemma}[Rough Gronwall's lemma]
	\label{lem:gronwall}
	Let $E\colon[0,T]\to \R_+$ be a path such that there exist constants $\kappa,\ell>0,$ a super-additive map $\varphi $ and a control $\omega $ such that:
	\begin{equation}\label{rel:gron}
	\delta E_{s,t}\leq \left(\sup_{s\leq r\leq t} E_r\right)\omega (s,t)^{\kappa }+\varphi (s,t)\, ,
	\end{equation}
	for every $s\le t \in [0,T]$ under the smallness condition $\omega (s,t)\leq \ell$.
	
	Then, there exists a constant $\tau _{\kappa ,\ell}>0$ such that
	\begin{equation}
	\label{concl:gron}
	\sup_{0\leq t\leq T}E_t\leq \exp\left(\frac{\omega (0,T)}{\tau _{\kappa ,\ell}}\right)\left[E_0+\sup_{0\leq t\leq T}\left|\varphi (0,t)\right|\right].
	\end{equation}
\end{lemma}
\subsection{Useful inequalities.}
\subsubsection{On Poincaré and Poincaré-Wirtinger inequalities.}
We recall the classical Poincaré inequality 
\begin{theorem}\label{teo:poincare}(Poincaré inequality) Let $D\subset \mathbb{R}^n$ be an open bounded subset. Then for every $p\in [1,\infty)$ there exists a constant $C\equiv C(p,D,n)>0$, depending only on $p,D,n$, such that
\begin{align*}
\|v\|_{L^p(D)}\leq C \|\nabla v\|_{L^p(D)}\, ,\quad \forall v\in W^{1,p}_0(D)\, .
\end{align*}
\end{theorem}
The application of the Poincaré's inequality is restricted to functions which are null on the boundary in the sense of the trace. The Poincaré-Wirtinger inequality is an extension of the Poincaré's inequality to the whole $W^{1,p}(D)$ (see e.g. \cite[Corollary 5.4.1]{poincare_wirtinger}), which we recall.
\begin{theorem}\label{teo:Poin_wirt}
Let $D\subset\mathbb{R}^n$ be an open connected bounded subset with $C^1$ boundary. Then there exists a constant $C\equiv C(p,D,n)>0$, depending only on $p,D,n$,  such that for all $v\in W^{1,p}(D)$
\begin{align*}
\bigg\|v-\frac{1}{|D|}\int_{D}v(y)\dd y\bigg\|_{L^p(D)}\leq C \|\nabla v\|_{L^p(\Omega)}\, .
\end{align*}
\end{theorem}
\subsubsection{Interpolation inequalities on a one dimensional bounded domain.}
Recall that from the Gagliardo-Nirenberg-Sobolev inequality on a bounded one dimensional domain, the following inequality holds.
\begin{lemma}\label{lemma:interp_ladyz}
Let $D\subset \mathbb{R}$ be a bounded connected open domain with $C^1$ boundary.
Then for all $v\in W^{1,2}(D)$ there exists a constant $C>0$ depending on the domain such that it holds
\begin{align*}
\|v\|_{L^4(D)}\leq C \|v\|^{\frac{1}{4}}_{H^1(D)}\|v\|^{\frac{3}{4}}_{L^2(D)}\, .
\end{align*}
\end{lemma}
\begin{proof}
The assertion follows from e.g. \cite[Theorem 5.8]{Adams_Fournie} with $n=1$, $q=4$, $m=1$, $p=2$.
\end{proof}
We recall also the Agmon's interpolation inequality in one dimension.
\begin{lemma}
Let $D\subset \mathbb{R}$ be an open bounded domain with $C^1$ boundary, then there exists a constant $C>0$ such that for all $v\in H^1(D)$,
\begin{align*}
	\|v\|_{L^\infty(D)}\leq C\|v\|_{H^1(D)}^{1/2}\|v\|_{L^2(D)}^{1/2}\, .
\end{align*}
\end{lemma}

\subsection{Global in time existence and uniqueness of a solution if $\partial_x h_1=0$ (or $\partial_x h_2=0$) and in absence of anisotropic energy.}\label{sec:ex_uniq_global}
In the particular case of constant space component of the noise and in absence of anisotropic energy, it is possible to show global in time existence and uniqueness of a solution to \eqref{LLG}. In particular, as a consequence of the orthogonality of the noise, the estimates hold pathwise and the regularity of the moments of the initial condition passes directly to the $L^\infty(H^1(\mathbb{S}^2))$ norm of the solution. Assume that $u^0\in H^1(\mathbb{S}^2)$, then the following inequalities hold at the level of the approximations and for the solution $\mathbb{P}$-a.s.
\begin{align}\label{eq:energy_global}
\sup_{0\leq t\leq T} \|\partial_x u_t\|^2_{L^2}+\int_{0}^{T}\|u_r\times \partial_x^2 u_r\|^2_{L^2}\dd r\leq \|\partial_x u^0\|^2_{L^2}\,,
\end{align}
\begin{align}\label{eq:energy_global_2}
\int_{0}^{T}\|\partial_x u_r\|^4_{L^4} \dd r \lesssim \sup_{0\leq t\leq T}\|\partial_x u_t\|^3_{L^2}\int_{0}^{T}\|u_r\times \partial_x^2 u_r\|_{L^2}\dd r\leq \|\partial_x u^0\|^4_{L^2}\, ,
\end{align}
for $T<1$, where we employed Jensen's inequality in \eqref{eq:energy_global_2} to obtain that
\begin{align*}
\left(\int_{0}^{T}\|u_r\times\partial_x^2 u_r\|_{L^2}\dd r\right)^2\leq \frac{1}{T}\left(\int_{0}^{T}\|u_r\times\partial_x^2 u_r\|^2_{L^2}\dd r\right)\leq \int_{0}^{T}\|u_r\times\partial_x^2 u_r\|^2_{L^2}\dd r\leq \|\partial_x u^0\|^2_{L^2}\, .
\end{align*}
We aim to get existence and uniqueness on the whole time interval $[0,+\infty)$: since the proof of uniqueness on each compact interval relies on the Gronwall's Lemma, we can not use the same proof. We employ a classical argument, where we paste intervals $[0,T]$, for fixed $T<1$, and cover the whole $[0+\infty)$. On the first interval $[0,T]$, there exists a unique strong solution to \eqref{LLG} starting in $u^0$ and ending in $u_T$ which satisfies the energies \eqref{eq:energy_global} and \eqref{eq:energy_global_2}. Now we consider $u_T$ to be the new initial condition on the interval $[T,2T]$, then there exists a unique solution $u$ on $[T,2T]$ starting at $u_T$ and ending in $u_{2T}$ such that
\begin{align*}
\sup_{T\leq t\leq2 T} \|\partial_x u_t\|^2_{L^2}+\int_{T}^{2T}\|u_r\times \partial_x^2 u_r\|^2_{L^2}\dd r\leq \|\partial_x u_T\|^2_{L^2}\leq \|\partial_x u^0\|^2_{L^2}\, ,
\end{align*} 
where the last inequality follows as a consequence of \eqref{eq:energy_global}. Similarly, by employing \eqref{eq:energy_global_2}, we establish
\begin{align*}
\int_{T}^{2T}\|\partial_x u_r\|^4_{L^4} \dd r \lesssim \sup_{T\leq t\leq 2T}\|\partial_x u_t\|^3_{L^2}\int_{T}^{2T}\|u_r\times \partial_x^2 u_r\|_{L^2}\dd r\leq \|\partial_x u_T\|^4_{L^2}\leq \|\partial_x u^0\|^4_{L^2}\, .
\end{align*}
This procedure can be iterated on each interval $[(n-1)T,nT]$, for $n\in \mathbb{N}$ and leads to existence and uniqueness of a pathwise global solution to \eqref{LLG}. Under the assumption that the initial condition $u^0\in \mathcal{L}^4(\Omega; H^1(\mathbb{S}^2))$, we observe that the global in time solution $u$ to \eqref{LLG} belongs to the spaces $$\mathcal{L}^2(\Omega; L^\infty([0,+\infty);H^1)\cap L^2([0,+\infty);H^2))\cap\mathcal{L}^4(\Omega;L^\infty(H^1))$$
(the last bound comes from the pathwise estimate in \eqref{eq:energy_global}). Note that if $u^0\in \mathcal{L}^k(\Omega; H^1(\mathbb{S}^2))$, it also holds that $\mathcal{L}^k(\Omega;L^\infty(H^1))$, for all $k\geq 4$.

\subsection{Spherical Brownian motion: existence, uniqueness and unique invariant measure.}\label{sec:spherical_BM}
The content of this section is not new in the literature: we refer for instance to Chapter V of \cite{Ikeda_Watanabe}, to \cite{BM_hypersurface} or to \cite{Hsu_stoch_manifold}. We sketch an alternative proof of existence and uniqueness of the Brownian motion with values on the sphere using the same techniques of this paper, whereas we follow \cite{BM_hypersurface} to show that the process is a diffusion. 

We describe the evolution of a Brownian motion on a sphere by
\begin{align}\label{eq:eq_sph_BM}
B_t=B_s+\int_{s}^{t}B_r\times \circ\dd W_r\, ,
\end{align}
with initial condition $B_0\in \mathbb{S}^2$. The existence and uniqueness of the solution $B$ to \eqref{eq:eq_sph_BM} follows from the classical rough path theory, indeed it is a linear equation: uniqueness holds pathwise and, from the continuity of the It\^o-Lyons map, the solution is also adapted (this trivially implies a large deviations result and a support theorem). It also follows that $B$ lies for every $t>0$ on the sphere $\mathbb{S}^2$, provided the initial condition lies on the sphere. The equation is linear and therefore continuous with respect to the initial condition in the euclidean norm. In order to conclude existence of an invariant measure, we can employ again Krylov-Bogoliubov, provided we prove tightness of $(\mu_{T_n})_n$ defined by
\begin{align*}
\mu_{T_n}(B^C_R)=\frac{1}{T_n}\int_{0}^{T_n}\mathbb{P}(|B_r|^2_{\mathbb{R}^3}>R)\dd r\, .
\end{align*}
Since the solution $B$ lies on a sphere, we observe that $\mu_{T}(B^C_R)=1$ for $R\in [0,1]$ and $\mu_{T}(B^C_R)=0$ or $R>1$. This implies that $\lim_{R\rightarrow +\infty}\mu_{T}(B^C_R) =0$, thus $(\mu_{T})_T$ is tight.  Thus $B$ admits at least one invariant measure $\mu_{B}$ and \eqref{eq:eq_sph_BM} admits a stationary solution. 

\paragraph{Unique invariant measure and recurrence}
With similar steps as in Lemma \ref{lemma:lemma_sup_grad}, we observe that the It\^o formulation of \eqref{eq:eq_sph_BM} has the form
\begin{align*}
B_t=B_s-\int_{s}^{t}B_r \dd r+\int_{s}^{t}B_r\times\dd W_r\, 
\end{align*}
and we are in the context of \cite{BM_hypersurface}.  The authors in \cite{BM_hypersurface} prove that the generator of the semigroup of \eqref{eq:eq_sph_BM} is a spherical Laplacian (Laplace-Beltrami operator on the sphere $\mathbb{S}^2$): thus the solution process $B$ to \eqref{eq:eq_sph_BM} is a Brownian motion with values on the sphere. We recall briefly the proof. For a map in $C^2(M;\mathbb{R})$ we can apply It\^o's formula and obtain
\begin{align*}
f(B_t)=f(B_s)-\int_{s}^{t}\nabla_X f(B_r) \cdot B_r \dd r+ \frac{1}{2}\int_{s}^{t}	\mathrm{tr}[\gamma(B_r)^T\nabla^2 f(B_r) \gamma (B_r)] \dd r+\int_{s}^{t} \nabla_X f(B_r) \gamma(B_r) \dd W_r\, ,
\end{align*}
where $\gamma(x)=x\times\cdot$ for  all $x\in \mathbb{R}^3$. In particular, we observe that
\begin{align*}
\mathrm{tr}(\gamma(B)^T\nabla^2 f(B) \gamma (B))=\Delta f-\sum_{i=1}^{3}B^i(B\cdot [\partial_{i,1}f,\partial_{i,2}f,\partial_{i,3}f])\, .
\end{align*}
From a geometrical point of view, $D_Vf:=\nabla f(B)-B(B\cdot \nabla f)$ and $\Delta_Vf:=\mathrm{tr}(D^2_Vf)$ is the Laplace-Beltrami operator, which defines a diffusion operator on the sphere. In other terms, the process $B$ is a diffusion with generator $\Delta_{\mathbb{S}^2}$. Since $B$ is a Brownian motion on $\mathbb{S}^2$, it follows from Chapter V, Theorem 4.6 (i) - (iii) in \cite{Ikeda_Watanabe} that there exists a unique invariant probability measure $\mu$ of the form
\begin{align*}
\mu[\dd v]=\frac{\exp(-F(v))\dd v}{\int_{\mathbb{S}^2}\exp(-F(z))\dd z}\, ,
\end{align*}
such that $dF=0$: hence in this case $F(x)\equiv C\in [0,+\infty)$.
The  process $B$ is recurrent on $\mathbb{S}^2$, namely $\mathbb{P}(B_t\in A)>0$ for every $t>0$ and for every open subset $A\subset \mathbb{S}^2$: indeed it is a Brownian motion on a compact manifold (Corollary 4.4.6 in \cite{Hsu_stoch_manifold}).

	\bibliographystyle{plain}

\end{document}